\documentclass[a4paper,11pt]{amsart}
\long\outer\def\COMMENT#1{}
\usepackage[utf8x]{inputenc} 
\usepackage[final]{graphicx}
\usepackage{amsmath}
\usepackage{mathbbol} 
\usepackage{amssymb}
\usepackage{euscript} 
\usepackage{mathrsfs} 
\usepackage{bbm}
\usepackage[all]{xy} \usepackage[colorlinks=true]{hyperref}
\usepackage{stmaryrd} \usepackage{wasysym} \usepackage{bussproofs}
\usepackage{cmll} \EnableBpAbbreviations
\DeclareMathAlphabet{\mathpzc}{OT1}{pzc}{m}{it}
\DeclareMathAlphabet{\mathpzc}{OT1}{pzc}{m}{it}
\theoremstyle{plain}
\newtheorem{theorem}{Theorem}[section]
\newtheorem{proposition}[theorem]{Proposition}
\newtheorem{lemma}[theorem]{Lemma}

\newtheorem{definition}[theorem]{Definition}

\newtheorem{remark}[theorem]{Remark}

\newtheorem{remarks}[theorem]{Remarks}

\numberwithin{equation}{section} 

\newcommand{\catset}{{\mathbf{Set}}}
\newcommand{\qtext}[1]{\quad\text{#1}\quad}
\newcommand{\qqtext}[1]{\qquad\text{#1}\qquad}
\newcommand{\cathpo}{\mathbf{HPO}} 
\newcommand{\adj}{\dashv}
\newcommand{\ent}{\vdash} 
\newcommand{\qdot}{\,.\hspace{1mm}}
\newcommand{\equi}{\Leftrightarrow}
\newcommand{\op}{\mathsf{op}} 
\newcommand{\setopto}{\catset^\op\to}
\newcommand{\trip}{{\EuScript{P}}}
\newcommand{\prop}{{\mathsf{Prop}}}
\newcommand{\triptr}{{\mathsf{tr}}}
\newcommand{\kkpi}{\mathsf{\Pi}}

\newcommand{\kkqp}{\mathsf{QP}}

\newcommand{\push}{\mathclose\cdot} 

\newcommand{\orth}{\mathrel\bot} 
\newcommand{\comk}{\mathsf{k}} 
\newcommand{\coms}{\mathsf{s}}
\newcommand{\coma}{\mathsf{a}}
\newcommand{\comd}{\mathsf{d}}
\newcommand{\come}{\mathsf{e}} 
\newcommand{\comb}{\mathsf{b}}
\newcommand{\comw}{\mathsf{w}} 
\newcommand{\comc}{\mathsf{c}}
\newcommand{\comp}{\mathsf{p}} 
\newcommand{\compz}{\mathsf{p_0}}
\newcommand{\compo}{\mathsf{p_1}}

\newcommand{\imp}{\Rightarrow}
\newcommand{\subs}{\subseteq} 
\newcommand{\sups}{\supseteq}
\newcommand{\ad}[1]{\emph{Ad #1 --}}
\newcommand{\eincl}{\hookrightarrow}

\newcommand{\power}{\mathcal{P}} 
\newcommand{\pobot}{\power_\bot}
\newcommand{\msep}{\mathrel|\,} 
\newcommand{\setof}[2]{\{#1\msep #2\}} 

\newcommand{\defequi}{:\Leftrightarrow}

\newcommand{\catord}{\mathbf{Ord}}
\newcommand{\setord}{\catset^\op\to\catord}
\newcommand{\adcl}{\mathopen\downarrow}
\newcommand{\aucl}{\mathopen\uparrow}

\newcommand{\id}{\mathrm{id}}
\newcommand{\fifa}{\EuScript{A}} 
\newcommand{\fifc}{\EuScript{C}} \newcommand{\fifd}{\EuScript{D}}
\newcommand{\fife}{\EuScript{E}}

\newcommand{\fibc}{\mathscr{C}} 
\newcommand{\catcat}{{\mathbf{Cat}}}
\newcommand{\catslat}{\mathbf{SLat}}

\newcommand{\catiord}{\mathbf{IOrd}}
\newcommand{\pullbackcorner}[1][dr]{\save*!/#1-1.2pc/#1:(-1,1)@^{|-}\restore}
\newcommand{\rapp}{\mathrm{app}} \newcommand{\rimp}{\mathrm{imp}}
 \newcommand{\comt}{\mathsf{t}}
\newcommand{\comf}{\mathsf{f}} \newcommand{\comi}{\mathsf{i}}
\newcommand{\qdefequi}{\quad:\Leftrightarrow\quad}

\numberwithin{equation}{section}

\newcommand{\rpaxk}{(PK)}
\newcommand{\rpaxs}{(\hspace*{1pt}P\hspace*{.5pt}S\hspace*{.5pt})}
\newcommand{\rpaxc}{(PC)} \newcommand{\rpaxha}{(\hspace*{.5pt}PA)}
\newcommand{\rpaxe}{(PE)}

\newcommand{\cord}{{(C,\leq)}} \newcommand{\dord}{{(D,\leq)}}
\newcommand{\eord}{{(E,\leq)}}
\newcommand{\cs}{{\bf \mathsf {{s}}\,}}
\newcommand{\ck}{{\bf \mathsf {{k}}\,}}

\newcommand{\cCC}{{\bf \mathsf {\scriptstyle{CC}\,}}}
\newcommand{\ce}{{\bf \mathsf {e\,}}}
\newcommand{\cb}{{\bf \mathsf {b\,}}}
\newcommand{\cc}{{\bf \mathsf {c\,}}}
\newcommand{\cS}{{\bf \mathsf {\scriptstyle{S}}\,}}
\newcommand{\cK}{{\bf \mathsf {\scriptstyle{K}}\,}}

\newcommand{\cB}{{\bf \mathsf {\scriptstyle{B}}\,}}
\newcommand{\cC}{{\bf \mathsf {\scriptstyle{C}}\,}}
\newcommand{\cI}{{\bf \mathsf {\scriptstyle{I}}\,}}
\newcommand{\cE}{{\bf \mathsf {\scriptstyle{E}}\,}}

\newcommand{\ioca}{\ensuremath{\mathcal{{}^IOCA}}}
\newcommand{\koca}{\ensuremath{\mathcal{{}^KOCA}}}
\newcommand{\oca}{\ensuremath{\mathcal{OCA}}}
\newcommand{\pca}{\ensuremath{\mathcal{PCA}}}

\begin{document}
\begin{abstract}
We consider different classes of combinatory structures related to
Krivine realizability. We show, in the precise sense that they give
rise to the same class of triposes, that they are equivalent for the
purpose of modeling higher-order logic.  We center our attentions in
the role of a special kind of Ordered Combinatory Algebras-- that we
call the \emph{Krivine ordered combinatory algebras} ($\koca$s)-- that
we propose as the foundational pillars for the categorical perspective
of Krivine's classical realizability as presented by Streicher in
\cite{kn:streicher}.

Our procedure is the following: we show that each of the considered
combinatory structures gives rise to an indexed preorder, and describe
a way to transform the different structures into each other that
preserves the associated indexed preorders up to equivalence. Since
all structures give rise to the same indexed preorders, we only prove
that they are triposes once: for the class of $\,\koca$\hspace*{1pt}s.

We finish showing that in $\koca$\hspace*{1pt}s, one can define
realizability in every higher-order language and in particular in
higher-order arithmetic.

\COMMENT{We consider 
different classes of combinatory
structures related to Krivine realizability. We then show that they
are equivalent for the purpose of modeling higher-order logic, in the
precise sense that they give rise to the same class of triposes.  We
emphasize the role of a special kind of Ordered Combinatory Algebras--
the \emph{Krivine ordered combinatory algebras} ($\koca$s)-- that we
propose as the foundational pillars for the categorical perspective of
Krivine's classical realizability as presented by Streicher in
\cite{kn:streicher}.

Our strategy is the following: we show how each type of combinatory
structure gives rise to an indexed preorder, and explain how to
transform the different types of structures into one another, in a way
that preserves the associated indexed preorder up to
equivalence. Since all structures give rise to the same indexed
preorders, we only have to prove that they are triposes once, and we
do this for $\koca$\hspace*{1pt}s.

Moreover, we show that in this kind of combinatory algebras:
$\koca$\hspace*{1pt}s, we can define realizability in every higher-order
language and in particular in higher-order arithmetic.}
\end{abstract}

\title[$\koca${\tiny s} and realizability] 
{Ordered combinatory algebras and realizability}
\author[W. Ferrer\and J. Frey \and M. Guillermo \and O. Malherbe \and
  A. Miquel]{Walter Ferrer Santos \and Jonas Frey \and Mauricio Guillermo
  \and 
Octavio Malherbe \and Alexandre Miquel}
\maketitle
\tableofcontents
\newpage
\section{Introduction}
Classical realizability was introduced in the mid 90's by Krivine
as a complete
reformulation of the principles of Kleene's (intuitionistic)
realizability (see \cite{Kleene}), to take into account the connection
between control 
operators and classical reasoning discovered by Griffin (in
\cite{Griffin}). 
Initially developed in the framework of classical second-order
Peano arithmetic (see \cite{KrivineStorageOp}), classical realizability was
quickly extended to 
Zermelo-Fraenkel set theory in \cite{kn:kr2001} using model-theoretic
constructions 
reminiscent both to  
the construction of generic extensions in forcing and 
to the construction of intuitionistic realizability models of 
intuitionistic set theories, see \cite{Myhill73}, \cite{Friedman73},
\cite{McCarty84}. 
In particular, Krivine showed in~\cite{kn:kr2003} how to interpret the
(classical) axiom 
of dependent 
choices in this framework.
More recently, he also showed in~\cite{kn:krreal} how to combine
classical realizability with 
the method of forcing (in the sense of Cohen), in the very spirit of
iterated forcing. 

Actually, Krivine's realizability (particularly in its most recent
developments) was mostly developed regardless to the
long-standing tradition of intuitionistic realizability.
And, as mentionned by Streicher in~\cite{kn:streicher}, it was dificult to
see how Krivine's work could fit into the structural approach to
realizability as initiated by Hylandin ~\cite{kn:hyland} and fully
described in~\cite{kn:vanOosbook}. 
One problem comes from the fact that the only realizability topos that
fulfils classical logic is the one based on the trivial partial
combinatory algebra (\pca) and thus,
equivalent to $\catset$, a fact that suggested 
for a long time that realizability and classical logic were incompatatible.

To resolve this paradox, Streicher proposed in~\cite{kn:streicher} a
categorical model for Krivine's 
realizability, still using the standard method that consists to
combine the construction of a realizability tripos with the well known
tripos-to-topos construction (see~\cite{kn:vanOosbook}). 
However, Streicher's
construction of the realizability tripos departs
from the standard construction from a \pca in several aspects. 

First, Streicher does not use a \pca, but a particular form of ordered
combinatory algebra (\oca), that is built from an abstract Krivine
structure ($\mathcal{AKS}$) that provides the computational
ingredients of Krivine 
realizability.  

Second, the elements of the considered \oca (induced by the underlying
$\mathcal{AKS}$) are not used as
realizers, but directly as truth values, using the fact that the
considered \oca\ has a meet-semilattice structure. In this
way, Streicher can skip the step that
consists to take the powerset to define truth values, and more
generally relations (as one would do if working with a \pca). 

A third ingredient of Streicher's construction is
the introduction of a specific notion of filter, to distinguish the
truth values that actually capture the notion of truth/provability.  
In practice, this filter is naturally defined from the
pole of the $\mathcal{AKS}$ and the corresponding notion of proof-like
terms.
Two warnings here
concerning notations: following the usual trends we use sometimes the
expression quasi--proof following the original
french \emph{quasi--preuve} instead of proof-like term; also the
reader should be aware that the notion of filter used in this context
is due to Hofstra (see~\cite{kn:hofstra2006}) and 
is different from the usual one. 

In this paper, we revisit Streicher's work by showing that his
construction can be performed working directly from a particular form
of \oca, 
which we call $\koca$, whose elements can be indifferently used as
realizers (or conditions) and as truth values, similarily to the
elements of a 
complete Boolean algebra in forcing. 
In particular, it should be clear to the reader that complete Boolean
algebras are particular cases of $\koca$s, and that in this case, the
general construction presented here amounts to the standard
construction of a Boolean tripos. So that the concept of $\koca$ can
be seen as the common denominator between classical realizability and
Cohen forcing. 
Moreover, to make more striking the comparison with the
standard approach of categorical models of 
realizability, we actually present the tripos construction starting from a
slightly more general structure of $\ioca$ that does not assume
anything about the logic being classical.  

\section{Streicher's Abstract Krivine Structures}

As motivation for the introduction of the concept of \emph{Krivine
  ordered combinatory algebras} ($\koca$\hspace*{1pt}s), we
recapitulate the definitions and basic ideas in \cite{kn:streicher},
regarding the notion of \emph{Abstract Krivine Structures} $\mathcal
{AKS}$.  These ideas were introduced by J.L. Krivine and reformulated
categorically by T. Streicher --see \cite{kn:kr2008} and
\cite{kn:streicher} respectively--.
\subsection{Realizability lattices} 
\begin{definition} A realizability lattice --abbreviated as $\mathcal
  {RL}$-- consists of:
\begin{enumerate}
\item A triple $(\Lambda,\Pi,\Perp)$ where $\Lambda$ and $\Pi$ are
  sets, \emph{of terms} and \emph{stacks} respectively, 
  and $\Perp \subseteq\Lambda \times \Pi$ is a subset (relation)
  ($\Lambda \times \Pi$ is the set of \emph{processes} and its
  elements are written $(t,\pi)$ or $(t,\pi)=t \star \pi$, moreover if
  $t \star \pi \in \Perp$, we write $t \perp \pi$, i.e. ``$t$ is
  orthogonal to $\pi$ or $t$ realizes $\pi$''). If $P \subset \Pi$
  and $t \perp \pi$ for all $\pi \in P$ we say that ``$t$ realizes $P$''
  and write $t \perp P$.
\item Define the following maps:
\begin{align*} (\quad)^{\perp}:\mathcal
P(\Lambda)&\xrightarrow{\hspace*{0.5cm}} \mathcal P(\Pi)\\ \Lambda
\supseteq L &\longmapsto L^{\perp}=\{\pi \in
\Pi|\,\, \forall t \in L, t \star \pi \in \Perp\}=\{\pi \in \Pi|\,\, L
\times \{\pi\} \subseteq \Perp\}\subseteq \Pi;
\end{align*}
\begin{align*} {}^{\perp}(\quad):\mathcal
  P(\Pi)&\xrightarrow{\hspace*{.5cm}} \mathcal P(\Lambda)\\ \Pi
  \supseteq P &\longmapsto {}^{\perp}P=\{t \in
  \Lambda|\,\, \forall \pi \in P, t \star \pi \in \Perp\}=\{t \in
  \Lambda|\,\, \{t\} \times P \subseteq \Perp\}\subseteq \Lambda.
\end{align*}
\item Define the following sets:
\[\mathcal P_{\perp}(\Lambda)=\{L \subseteq \Lambda:\,
{}^\perp(L^{\perp}) = L\} \subseteq \mathcal P(\Lambda)\quad,\quad
\mathcal P_{\perp}(\Pi)=\{P \subseteq \Pi:\, ({}^{\perp}P)^{\perp} =
P\} \subseteq\mathcal P(\Pi).\]
\end{enumerate}
\end{definition}
\begin{remark}\label{remark:initialrl}
\begin{enumerate}
\item The maps $L \rightarrow L^{\perp}$ and $P \rightarrow
  {}^{\perp}P$ are antitonic with respect to the order given by the
  inclusion of sets and for $P_i \subseteq\Pi$, $L_i \subseteq\Lambda,
  {i \in I}$ we have: $${}^{\perp}\big(\bigcap_{i \in I} P_i\big) \supseteq
  \bigcup_{i \in I} {}^{\perp}P_i\,,\,{}^{\perp}\big(\bigcup_{i \in I}
  P_i\big) = \bigcap_{i \in I} {}^{\perp}P_i;\big(\bigcap_{i \in I}
  L_i\big)^{\perp} \supseteq \bigcup_{i \in I}
  L_i^{\perp}\,,\,\big(\bigcup_{i \in I} L_i\big)^{\perp} = \bigcap_{i
    \in I} L_i^{\perp}.$$
\item For an arbitrary $L \in \mathcal P(\Lambda)$ and $P \in \mathcal
  P(\Pi)$, one has that ${}^{\perp}(L^{\perp}) \supseteq L$ and
  $({}^{\perp}P)^{\perp} \supseteq P$. For an arbitrary $L \in
  \mathcal P(\Lambda)$ and $P \in \mathcal P(\Pi)$, one has that
  $({}^\perp(L^{\perp}))^{\perp} = L^{\perp}$ and
  ${}^{\perp}(({}^\perp P)^\perp) = {}^{\perp}P$.
\item \label{item:completion} The maps $(\quad)^{\perp}:\mathcal
  P(\Lambda)\rightarrow \mathcal P(\Pi)$ and
  ${}^{\perp}(\quad):\mathcal P(\Pi)\rightarrow \mathcal P(\Lambda)$
  when restricted respectively to $\mathcal P_{\perp}(\Lambda)$ and
  $\mathcal P_{\perp}(\Pi)$ are order reversing isomorphisms inverse
  to each other.  An $\mathcal {RL}$ satisfies two strong
  completion properties.  If $\mathcal X$ is a subset of $\mathcal X
  \subseteq \mathcal
  P_\perp(\Pi)$, define: \[\operatorname{sup}(\mathcal
  X)=\Big({}^\perp\big(\bigcup\{P: P \in \mathcal X
  \}\big)\Big)^\perp\,\,,\,\,\operatorname{inf}(\mathcal
  X)=\bigcap\{P: P \in \mathcal X \}.\] In particular,
  $\operatorname{sup}(\mathcal X)\, \text{and}\,
  \operatorname{inf}(\mathcal X)$ are the supremum and infimum of the
  set $\mathcal X$ in $\mathcal P_\perp(\Pi)$ with respect to the
  order given by the inclusion of sets.  Moreover with respect to the
  order given by the inclusion, $\Lambda^\perp$ and $\Pi$\emph{;}
  ${}^\perp\Pi $ and $\Lambda $ are the minimal and maximal elements
  of $\mathcal P_{\perp}(\Pi)$ and $\mathcal P_{\perp}(\Lambda)$
  respectively.

\end{enumerate}
\end{remark} The only relevant structure at this point is the
\emph{lattice} structure in the sets $P_{\perp}(\Lambda)$ and
$P_{\perp}(\Pi)$, where we take the (set theoretical) inclusion as the
order the intersection as ``meet'' and the union followed by taking
double orthogonals as ``join''.
\subsection{The push map in a realizability lattice} In this section we
add a \emph{push map} to a realizability lattice, thus introducing the
first elements of a \emph{calculus} into our structure.
\begin{definition}\label{defi:maps}
\begin{enumerate}
\item A map $\operatorname{push}(t,\pi):(t,\pi) \mapsto t . \pi :
\Lambda \times \Pi \rightarrow \Pi$ defined in a realizability lattice
$(\Lambda,\Pi,\Perp)$, will be called a \emph{push map}.
\item For an $\mathcal{RL}$ with a push map, and for $L \subseteq
  \Lambda$, $P \subseteq \Pi$ we define, $$(L,P)\mapsto L \leadsto
  P:\mathcal P(\Lambda) \times \mathcal P(\Pi) \rightarrow \mathcal
  P(\Pi)$$ with $L \leadsto P=\{\pi \in \Pi: L.\pi \subseteq P\}
  \subseteq \Pi$ --called \emph{the right conductor of $L$ into
    $P$}. We consider also $(L,P)\mapsto L\cdot P:\mathcal P(\Lambda)
  \times \mathcal P(\Pi) \rightarrow \mathcal P(\Pi)$ where:
  $L \cdot P=\{t \cdot \pi:t \in L,\, \pi \in P\}$.
\end{enumerate}
\begin{remark} \label{remark:firstadj} Clearly, for $L$ and $P$ as
  above: $L \cdot P \subseteq 
  Q$ if and only if $P \subseteq L \leadsto Q$.
\end{remark}
\end{definition} The above constructions of $L \leadsto P$ and $L.P$
combined with the operators $(\quad)^\perp$ and ${}^\perp(\quad)$
yield natural binary operations in $\mathcal P_{\perp}(\Pi)$ that are
basic ingredients of the $\mathcal {OCA}$ associated to the $\mathcal
{AKS}$ \emph {\`a la Streicher}.

We define the following binary operations between subsets of $\Pi$.
\begin{definition} \label{defi:maps2}  Let $P,Q \subseteq \Pi$ then:
\begin{align}
  \label{eqn:maps1} P \circ Q:=~^\perp Q \leadsto P \subseteq P
  \circ_\perp Q:=(^\perp(^\perp Q \leadsto P))^\perp,\\
\label{eqn:maps2} P\Rightarrow Q:={}^\perp P \cdot Q \subseteq
P\Rightarrow_\perp Q:=({}^\perp({}^\perp P \cdot Q))^{\perp} \in
\mathcal P_{\perp}(\Pi).
\end{align}
\end{definition} Once the above definitions are established, we can
deduce a crucial ``half adjunction property'' relating the operations
$\circ_\perp$ and $\Rightarrow_\perp$ in $\mathcal P_\perp(\Pi)$.
\begin{theorem}\label{theorem:adjunction}\emph{[Half adjunction property]}
  Assume that $P,Q,R \in \mathcal P_{\perp}(\Pi)$. If $Q
  \Rightarrow_\perp\!\!R \subseteq P$, then $R \subseteq P
  \circ_{\perp}Q$. In particular: $P \subseteq (Q
  \Rightarrow_\perp\!\!P)\circ_{\perp} Q$.
\end{theorem}
\begin{proof} The following inclusions are equivalent: $Q
  \Rightarrow_\perp\!\!R \subseteq P$, $({}^\perp({}^{\perp}Q \cdot
  R))^{\perp} \subseteq P$, ${}^{\perp}Q \cdot R \subseteq P$ and $R
  \subseteq {}^\perp Q \leadsto P$. The last inclusion implies that $R
  \subseteq ({}^\perp({}^\perp Q\leadsto P))^{\perp} = P \circ_{\perp} Q$.
\end{proof} 

\subsection{Abstract Krivine structures}  Next, --compare with
\cite{kn:streicher}-- we complete the process of adding a calculus to
a realizability lattice to obtain the concept of \emph{Abstract
Krivine Structure} abbreviated as $\mathcal {AKS}$. For that, we
introduce the usual application map for terms, a store map from stacks
to terms, the combinators $\cK,\cS$, and a distinguished term $\cCC$
that is a realizer of Peirce's law. We introduce also a set of terms
that we call quasi--proofs and assume that the three combinators above
are quasi proofs.
\begin{definition} \label{defi:aks}An \emph{Abstract Krivine
    Structure} (frequently written as $\mathcal K$) consists of the
  following elements:
\begin{enumerate}
\item \label{item:aksdecuple} A realizability lattice with a push
$(\Lambda,\Pi,\Perp,\operatorname{push})$,
\item Functions 
\begin{enumerate}
\item $\operatorname{app}:\Lambda \times \Lambda \rightarrow \Lambda$
is a function: $(t,u) \mapsto \operatorname{app}(t,u)=tu$,
\item $\operatorname{store}:\Pi \rightarrow \Lambda$ is a function:
$\pi \mapsto \operatorname{store}(\pi)=k_\pi$,
\end{enumerate}
\item A set $\operatorname {QP} \subseteq \Lambda$ of
  ``quasi--proofs'', which is closed under application,
\item Elementary combinators $\cK,\cS,\cCC \in \operatorname {QP}$.

\item \label{item:aksaxioms} The above elements are subject to the
following axioms.  \newcounter{qcounter}
\begin{list}{(S\arabic{qcounter})}{\usecounter{qcounter}}
\item \label{item:first} If $t \perp s \cdot \pi$, then $ts \perp
\pi$.
\item If $t \perp \pi$, then for all $s \in \Lambda$ we have that $\cK
\perp t \cdot s \cdot \pi$.
\item If $tu (su) \perp \pi$, then $\cS \perp t \cdot s \cdot u \cdot
\pi$.
\item If $t \perp k_\pi \cdot \pi$, then $\cCC \perp t \cdot \pi$.
\item If $t \perp \pi$, then for all $\pi'\in \Pi$ we have that $k_\pi
\perp t \cdot \pi'$.
\end{list}
\end{enumerate}
\end{definition} Here --and in the rest of this paper-- the product--like
operations will not be associative and we assume that when parenthesis
are omitted, we associate to the left.

The elements of the structure above, named
as: \[\operatorname{store}:\pi \mapsto k_\pi :\Pi \rightarrow \Lambda
\quad \text{and} \quad \cCC \in \Lambda,\] have a very special role in
the sense they can be used to make the realizability theory
\emph{classical} as $\cCC$ realizes Peirce's law. In this sense in the
presence of the mentioned elements and the corresponding axioms (S4)
and (S5), the $\mathcal {AKS}$-- is \emph{classical}.

\begin{definition} \label{defi:maps1}For a general $\mathcal {AKS}$
  we introduce the following definitions: 
\begin{enumerate}
\item For $L,M \subseteq \Lambda$ we define $L \Rightarrow M=\{t \in
  \Lambda: tL \subseteq M\}$, 
\item For $P,Q \subseteq \Pi$, $P \diamond Q:=\big(({}^\perp P)
  ({}^\perp Q)\big)^{\perp} \in \mathcal P_{\perp}(\Pi)$.
\item $\cI= \cS\cK\cK,\,\cB=\cS(\cK\cS)\cK,\, \cE=\cS(\cK\cI) \in \operatorname{QP}$.
\end{enumerate}
 \end{definition}
\begin{lemma}\label{lem:circdiamond} In an $\mathcal{AKS}$, for $P,Q
  \in \mathcal P_\perp(\Pi)$, we have that condition \emph{(S1)} in
  Definition \ref{defi:aks}, \eqref{item:aksaxioms} implies any of the
  two equivalent conditions below.
\begin{enumerate}
\item $P \circ_{\perp} Q \subseteq ({}^\perp P {}^\perp Q)^\perp =P \diamond
Q$
\item If $t \perp P$ and $s \perp Q$, then $ts \perp P \circ_{\perp} Q$.
\end{enumerate}
\end{lemma}
\begin{proof}
  It is evident that the two conditions above are equivalent.
  Assuming (S1), we want to prove that for all $P,Q$ then: $\{\pi \in
  \Pi: {}^\perp Q.\pi \subseteq P\} \subseteq ({}^\perp P {}^\perp
  Q)^\perp$.

In other words we want to show that if $\pi \in \Pi$ is such that
${}^\perp Q.\pi \subseteq P$ then, for all $s \perp P,\, t \perp Q$ we
have that $st \perp \pi$. It is clear that from the hypothesis
${}^\perp Q.\pi \subseteq P$ and $s \perp P,\, t \perp Q$, that $s
\perp t.\pi$ and in this case the original condition (S1) implies that
$st \perp \pi$.
\end{proof}
Next, we deduce some consequences or equivalent formulations of the
basic axioms for an $\mathcal {AKS}$ in terms of elements of $\mathcal
P_\perp(\Pi)$ and the operations $\circ_{\perp}$,
$\rightarrow_{\perp}$, $\diamond$.
\begin{lemma}\label{lem:aks} For 
$P,Q,R\in \mathcal P_{\perp}(\Pi)$, $t,u, v\in\Lambda$, and $\pi\in\Pi$, we have:
\begin{enumerate}
\item \label{lem:aks-0}$t \perp (P\imp_{\perp} Q) \imp_{\perp} R$ if
  and only if $t \perp (P\imp Q) \imp R$. Also if $t \perp
  P\imp_{\perp} (Q \imp_{\perp} R)$ then $t \perp P\imp (Q \imp R)$.
\item\label{lem:aks-mp} $t\perp P\imp_{\perp} Q,\quad u\perp
P\quad\text{implies}\quad tu\perp Q$;
\item\label{lem:aks-k} $\cK\perp P\imp Q\imp P$;
\item\label{lem:aks-s} $\cS\perp (P\imp Q\imp R)\imp (P\imp Q)\imp
  P\imp R$;
\item\label{lem:aks-cc} $\cCC\perp ((P\imp_{\perp} Q)\imp_{\perp}
  P)\imp_{\perp} P$ or equivalently $\cCC\perp ((P\imp Q)\imp P)\imp
  P$;
\item\label{lem:aks-i-orth} $t\orth\pi$ implies $\cI\orth
t\push\pi$;
\item\label{lem:aks-i-real} $\cI\perp P\imp_{\perp} P$;
\item\label{lem:ask-borth} $t\orth uv\push\pi$ implies $\cB\orth
t\push u\push v\push\pi$;
\item\label{lem:aks-b-real} $\cB\perp(Q\imp R)\imp (P\imp Q)\imp
P\imp R$;
\item\label{lem:aks-e} $tu\orth\pi$ implies $\cE t\orth u\push\pi$.
\end{enumerate}
\end{lemma}
\begin{proof} 
  \ad{\ref{lem:aks-0}} This follows immediately from the fact that
  taking orthogonals three times is the same than taking
  them once. 

\medskip

  \ad{\ref{lem:aks-mp}} Take $t\perp (P \Rightarrow_{\perp}Q)$ and
  $u\in {}^{\perp} P$ and we want to prove that $tu \perp Q$ or
  equivalently that ${}^\perp(P\Rightarrow_{\perp}Q) ({}^\perp P)
  \subseteq {}^\perp Q$ or that $(P \Rightarrow_{\perp} Q) \diamond P
  \supseteq Q$. We have (Lemma \ref{lem:circdiamond}) that $(P
  \Rightarrow_{\perp} Q) \diamond P \supseteq (P \Rightarrow_{\perp}
  Q) \circ_\perp P \supseteq Q$. See Theorem \ref{theorem:adjunction}
  for the last inequality.

\medskip

\ad{\ref{lem:aks-k}} Let $t\in {}^\perp P$, $u\in {}^\perp Q$, $\pi\in
P$. We have to show that $\cK\orth t\push u\push \pi$. By (S2), it is
sufficient to show $t\orth \pi$, which follows from the definition of
${}^\perp P$.

\medskip

\ad{\ref{lem:aks-s}} Let $t\in{}^\perp(P\imp Q \imp R)$,
$u\in{}^\perp(P\imp Q)$, $v\in {}^\perp P$, $\pi\in R$.  Using
\ref{lem:aks-mp} we can deduce $tv(uv)\in {}^\perp R$ and thus
$tv(uv)\orth\pi$.  Axiom (S3) implies $\cS\orth t\push u\push
v\push\pi$, as required.

\medskip

\ad{\ref{lem:aks-cc}} Let $t\in{}^\perp((P\imp_{\perp} Q)\imp_{\perp}
P)$ and $\pi\in P$.  We have to show that $\cCC\orth t\push\pi$, and by
(S4) it is sufficient to show that $t\orth k_\pi\push\pi$. This
would follow from $k_\pi\in {}^\perp(P\imp_{\perp} Q)$, so it remains
to prove the latter.

Let $u\in {}^\perp P$, $\pi'\in Q$. We have to show that $k_\pi\orth
u\push\pi'$, and by (S5) it is sufficient to show that
$u\orth\pi$, which is true since $u\in {}^\perp P$ by assumption.

\medskip

\ad{\ref{lem:aks-i-orth}} $t\orth\pi$ implies $\cI\orth
t\push\pi$. Indeed, $t \perp \pi$ $\Rightarrow$ $\cK \perp t\cdot(\cK
t)\cdot \pi$ $\Rightarrow$ $\cK t \perp \cK t\cdot \pi$ $\Rightarrow$
$\cK t(\cK t) \perp \pi$ $\Rightarrow$ $\cS \perp \cK \cdot\cK \cdot
t\cdot \pi$$\Rightarrow$ $\cI = \cS \cK \cK \perp t \cdot \pi$.

\medskip

\ad{\ref{lem:aks-i-real}} It is clear that the assertion $\cI\perp
P\imp_{\perp} P$ is another formulation of \ref{lem:aks-i-orth}.

\ad{\ref{lem:aks-e}} The following chain of implications proves that
$tu\orth\pi$ implies $\cE t :=\cS(\cK(\cS\cK\cK))t \orth u\push\pi$.
\begin{eqnarray} \nonumber tu\perp \pi \Rightarrow  \cI=\cS\cK\cK \perp
tu \cdot \pi \Rightarrow \cK\perp\cS\cK\cK \cdot u \cdot tu \cdot \pi\Rightarrow
\cK(\cS\cK\cK)u(tu)\perp \pi \\ \nonumber \cK(\cS\cK\cK)u(tu)\perp \pi
\Rightarrow \cS\perp (\cK(\cS\cK\cK))\cdot t\cdot u \cdot \pi \Rightarrow
\cS(\cK(\cS\cK\cK))t \perp u \cdot \pi.
\end{eqnarray}
In the first implication we used the definition of $\cI$ in the second
the definition of $\cK$, in the third we used (S1) in the fourth the
definition of $\cS$ and in the last one, we used property (S1) again.  
\end{proof}
 \begin{remark}
Clauses \ref{lem:aks-mp}-\ref{lem:aks-cc} resemble the Hilbert style 
axiomatization of the implicational fragment of classical
propositional logic\COMMENT{(see Section~\ref{app:hilbert})}.  

Using this analogy, it is easy to show the following.

Assume that $\varphi[X_1,\dots,X_n]$ is a propositional formula built
up from propositional variables $X_1,\dots,X_n$ and implication. For
arbitrary subsets $P_1,\dots,P_n\subs\Pi$ denote by
$\varphi[P_1,\dots,P_n]\subs\Pi$ the evaluation of
$\varphi[X_1,\dots,X_n]$ where $P_1,\dots,P_n$ are substituted for the
variables, and implication is interpreted by the operation $\imp$ from
Definition~\ref{defi:maps2}.  If $\varphi[X_1,\dots,X_n]$ is provable
in the Hilbert calculus, \COMMENT{system described in
  Appendix~\ref{app:hilbert},} then 
$\varphi[P_1,\dots,P_n]$ contains a quasi-proof, namely the element of
$\operatorname{QP}$ obtained by evaluating the proof term of
$\varphi[X_1,\dots,X_n]$ in $\Lambda$.
\end{remark}

Next we list some properties satisfied by any combinator that
satisfies the $\cS \eta$ rule (i.e. the rule \ref{lem:aks-e}) in
particular by $\cE$.
\begin{lemma}\label{lem:setheoSeta}
\begin{enumerate}
\item If a combinator $\widehat{\cE}$ satisfies the $\cS \eta$ rule
--i.e. $ts \perp \pi \Rightarrow \widehat{\cE}t \perp s\cdot \pi$--
then it satisfies any of the equivalent assertions that follow.
\begin{equation}\label{eqn:setados} \text{If}\,\, P,Q \in \mathcal
  P_\perp(\Pi)\,\, \text{then}\,\, P \diamond Q \subseteq \{\pi \in \Pi:
  \widehat{\cE}({}^\perp P) \subseteq {}^\perp({}^\perp Q\cdot  \pi)\}=\{\pi
  \in \Pi: (\widehat{\cE}({}^\perp P))^\perp \supseteq ({}^\perp
  Q\cdot  \pi)\}.
\end{equation}
\begin{equation}\label{eqn:setatres} \text{If}\,\, P,Q \in \mathcal
  P_\perp(\Pi)\,\, \text{then}\,\, \widehat{\cE}({}^\perp P) \subseteq
  {}^\perp\big({}^\perp Q\cdot (P \diamond Q) \big).
\end{equation}
\begin{equation}\label{eqn:seta3.1} \text{If}\,\, R \subseteq (P
  \diamond Q),\, \text{with}\,\, P, Q, R \in P_\perp(\Pi)\,\,
  \text{then}\,\, \widehat{\cE}{}(^\perp P) \subseteq {}^\perp({}^\perp
  Q\cdot  R).
\end{equation}
\item If a combinator $\widehat{\cE}$ satisfies the $\cS \eta$ rule
then --in the notations of Definition \ref{defi:maps}--, the following
assertions hold.
\begin{equation}\label{eqn:setacuatro} \text{If}\,\, P,Q \in \mathcal
  P_\perp(\Pi)\,\, \text{then}\,\, \widehat{\cE}({}^\perp({}^\perp P
  \cdot  Q)) \subseteq {}^\perp({}^\perp P \cdot  Q)\,\, \text{or
    equivalently}\,\, \widehat{\cE}({}^\perp(P \Rightarrow_{\perp}{} Q)) \subseteq
  {}^\perp(P \Rightarrow_{\perp} Q).
\end{equation}
\begin{equation}\label{eqn:setauno}\big(t({}^\perp P)\big)^\perp
  \subseteq \{\pi \in \Pi: (\widehat{\cE}t)^\perp \supseteq ( {}^\perp
  P\cdot  \pi)\} \subseteq ({}^\perp\{\pi \in \Pi: (\widehat{\cE}t)^\perp
  \supseteq ( {}^\perp P\cdot  \pi)\})^{\perp}=(\widehat{\cE}t)^\perp \circ_{\perp}{}P.
\end{equation}
\begin{equation}\label{eqn:seta3.5}\text{If}\,\, P,Q \in \mathcal
  P_\perp(\Pi)\,\, \text{then}\,\,(P \diamond Q) \subseteq
  (\widehat{\cE}({}^\perp P))^\perp \circ_{\perp}{} Q.
\end{equation}
\item \label{item:adjunctorpaks}If $P \in \mathcal P_\perp(\Pi)$ and
  $\widehat{\cE}$ are as above, then: $(\widehat{\cE}(^\perp P))^\perp
  \subseteq (\widehat{\cE}\widehat{\cE})^\perp \circ_{\perp}{} P$.
\end{enumerate}
\end{lemma}
\begin{proof}
\begin{enumerate}
\item It is clear that the inclusions
\eqref{eqn:setados},\eqref{eqn:setatres},\eqref{eqn:seta3.1} are all
equivalent.

Now we prove that if $\forall s,t,\pi$, $ts \perp \pi \Rightarrow
\widehat{\cE}t \perp s\cdot \pi$ then the inclusion
\eqref{eqn:setados} holds.

Take $\pi \in P \diamond Q$. This means that for all $t \perp P, s
\perp Q$, $ts \perp \pi$ and then $\widehat{\cE} t \perp s \cdot
\pi$. This means that the inclusion \eqref{eqn:setados} holds.

\item The chain of inclusions below, show that if the combinator
$\widehat{\cE}$ satisfies the $\cS \eta$ rule, then
\eqref{eqn:seta3.5} holds.

The other results are proved similarly.
\[P \diamond Q \subseteq \{\pi \in \Pi: (\widehat{\cE}({}^\perp
P))^\perp \supseteq ({}^\perp Q\cdot  \pi)\} \subseteq \big({}^\perp\{\pi
\in \Pi: (\widehat{\cE}({}^\perp P))^\perp \supseteq ({}^\perp
Q\cdot  \pi)\}\big)^\perp= (\widehat{\cE}({}^\perp P))^\perp \circ_{\perp} Q.\]
Note that the first inclusion is just \eqref{eqn:setados}.
\item The proof follows by substitution of $t$ by $\widehat{\cE}$ in
\eqref{eqn:setauno}.
\end{enumerate}
\end{proof} The theorem that follows recovers (partially) ``the
other half'' of the half adjunction property of Theorem
\ref{theorem:adjunction}.
\begin{theorem} \label{theo:adjunctionconverse} Let $P,Q,R \in \mathcal
P_\perp(\Pi)$. If \[R \subseteq P \circ_{\perp} Q \,\,\text{then}\,\, (Q
\Rightarrow_{\perp} R) \subseteq (\cE{}(^\perp P))^\perp \subseteq
(\cE\cE)^\perp \circ_{\perp}{} P.\]
\end{theorem}
\begin{proof} As $R \subseteq P \circ_{\perp} Q \subseteq P \diamond
  Q= ({}^\perp P {}^\perp Q)^\perp$ --see Lemma \ref{lem:circdiamond},
  we have that ${}^\perp Q \cdot R \subseteq {}^\perp Q\cdot ({}^\perp
  P {}^\perp Q)^\perp$ and ${}^\perp({}^\perp Q \cdot R) \supseteq
  {}^\perp({}^\perp Q \cdot ({}^\perp P {}^\perp Q)^\perp)$. By
  equation \eqref{eqn:setatres} we have that $\cE{}(^\perp P)
  \subseteq {}^\perp\big({}^\perp Q\cdot ({}^\perp P {}^\perp Q
  )^\perp \big) \subseteq {}^\perp({}^\perp Q \cdot R)= {}^\perp(Q
  \Rightarrow_{\perp}{} R)$. Taking orthogonals we obtain the first
  inclusion. The other inclusion is just Lemma \ref{lem:setheoSeta},
  \eqref{item:adjunctorpaks}, since by Lemma \ref{lem:aks},
  \eqref{lem:aks-e} $\cE$ satisfies the $\cS \eta$--rule.
\end{proof}
\begin{definition}\label{defi:adjunctor} In a $\mathcal{AKS}$ as above, the
combinator $\cE\cE \in \operatorname{QP}$ is called an adjunctor.
\end{definition}

\section{Implicative and Krivine ordered combinatory algebras}
In \cite{kn:streicher}, Streicher presented a construction of
an \emph{ordered combinatory algebra ($\oca$)}
(see~\cite[Section~1.8]{kn:vanOosbook}) out of an
$\mathcal {AKS}$ , from which he constructed a tripos whose predicates
are functions with values in the $\oca$.  This construction does not
give rise to a tripos in general, but only for some $\oca$s -- in
particular for those induced by $\mathcal {AKS}$ 's. The notion of
\emph{implicative ordered combinatory algebra} abbreviated as $\ioca$,
is an axiomatization of the additional structure that is necessary on
an $\oca$ to guarantee that the induced indexed preorder is a
tripos. The tripos will be classical in case the $\ioca$ has an
additional combinator, called $\comc$, that realizes Peirce's law. 

In this section we fix our attention on \emph{implicative ordered
  combinatory algebras}: $\ioca$\hspace*{1pt}s and the modification
consisting in adding the combinator $\comc$ that produces a
\emph{Krivine ordered
  combinatory algebra}: $\koca$.  The main features
added to the usual structure of an ordered combinatory algebra
--compare with \cite{kn:hofstra2006}-- are the following: a) we assume
the existence of a distinguished element, that we call
\emph{adjunctor}; b) we assume that the $\ioca$ is
$\operatorname{inf}$ --complete; c) we have an implication mapping
denoted as $\rightarrow$.  These additions are present in the
$\mathcal{OCA}$s that come from $\mathcal{AKS}$s, and will be crucial
ingredients in the construction of the associated tripos, that we
build up directly from the $\ioca$ --compare with
\cite{kn:streicher}--. See for example \cite{kn:partordcombalg} and
\cite{kn:vanOosbook} for the standard approach to the subject.

\subsection{Ordered combinatory algebras}

\begin{definition}
\begin{enumerate}
\item An \emph{ordered combinatory algebra} ($\oca$) is a quintuple
  $\mathcal A=(A,\leq,\rapp,\comk,\coms)$ --written frequently as
  $A$-- where $(A,\leq)$ is a partial order,
\[ \rapp: A\times A\to A, \qquad (a,b)\mapsto ab
\] is a monotone function, and $\comk,\coms$ are elements of $A$
satisfying
\begin{enumerate}
\item $\comk a b \leq a$
\item $\coms a b c\leq ac(bc)$
\end{enumerate} for all $a,b,c\in A$.
\item A \emph{filter} in an $\oca$ --called $A$-- is a subset
  $\Phi\subs A$ which contains $\coms$ and $\comk$ and is closed under
  application. A pair $(\mathcal A,\Phi)$ is called a filtered $\oca$.
\end{enumerate}
\end{definition}
\begin{remark} Here --and in the rest of this paper-- the product--like
operations will not be associative and we assume that when parenthesis
are omitted, we associate to the left.
\end{remark}
In what follows, we will recall how to program
directly in this $\mathcal{OCA}$, using the standard codifications in
the combinatory algebras.
\begin{definition} Let $A$ be an $\oca$ and take a denumerable set of
\emph{variables}: $\mathcal{V}=\{x_1,x_2,\cdots\}$ and consider
$A(\mathcal{V})$ --called \emph{the set of terms in $A$}-- that is the
set of formal expressions given by the following grammar: $p_1, p_2
::= a \quad |\quad x\quad |\quad p_1 p_2$ where $a\in A$ and
$x\in\mathcal{V}$. As usual $A(x_1,\dots,x_k)$ is the set of terms in $A$
containing only the variables $x_1,\cdots,x_k$.  One can naturally
extend the order in $A$ to an order in $A(\mathcal V)$ in such a way
that: if $p_1 \ R \ p_2, q_1 \ R \ q_2\in A(\mathcal{V})$ then $p_1q_1
\ R \ p_2q_2$ and if $p_1, p_2\in A(\mathcal{V})$ then $\ck p_1 p_2 \
R \ p_1$ and if $p_1, p_2, p_3\in A(\mathcal{V})$ then $\cs p_1 p_2
p_3 \ R \ p_1 p_3 (p_2 p_3)$.
\end{definition}

If we define an equivalence relation $\equiv_R$ on $A(\mathcal{V})$
as: $p \equiv_R q$ iff $p \ R \ q$ and $q \ R \ p$, the order $R$ can
be factored to a partial order in the quotient
$A[\mathcal{V}]:=A(\mathcal{V})/\equiv_R$ --called also the \emph{the
set of terms} in $A$--.

If $p_1, p'_1, p_2, p'_2$ are terms such that $p_1\equiv_R p'_1$ and
$p_2\equiv_R p'_2$ then $p_1 p_2 \equiv_R p'_1 p'_2$. Hence, the
application can be defined in the quotient and is written as $p_1
\star p_2$.

Hence, $(A[\mathcal{V}], R, \star)$ is an $\mathcal{OCA}$ and $(A,
\leq, \operatorname{app})$ is a sub-$\mathcal{OCA}$ of
$(A[\mathcal{V}], R, \star)$.
 
It is customary to denote the relation $R$ as $\leq$ and the
operation~$\star$ as~$\circ$ or as the concatenation of the
factors. In this situation we say that $(A[\mathcal{V}], \leq, \circ)$
is an extension of $(A, \leq, \circ)$.

The following result is well known.
\begin{theorem}[Combinatory completeness] \label{theo:calculusinA} For
  any finite set of variables $\{x_1,\cdots,x_k,y\} $, there is a
  function $\lambda^*y:A[x_1,\cdots,x_k,y] \rightarrow A[x_1,\cdots,
  x_k]$ satisfying the following property: \[\text{If}\quad t \in
  A[x_1,\cdots,x_k,y]\,\,,\,\,\text{and}\quad u \in A[x_1,\cdots,x_k]
  \quad \text{then} \quad (\lambda^*y (t))\circ u \leq
  t\{y:=u\}.\]  Moreover if $X \subseteq A$ is an arbitrary subset and
  $t$ is a term with all its coefficients in $X$, then $\lambda^*y(t)$
  is a term with all its coefficients in $\langle X \rangle$, the
  closure of $X$ by application. In
  particular if all the coefficients of $t$ are in the filter $\Phi$,
  then $\lambda^*y(t)$ is a polynomial with all the coefficients in
  $\Phi$. Occasionally we write $\lambda^*y(t)=\lambda^*y.t$.
\end{theorem} 
\begin{proof} The function $\lambda^*y$ is defined recursively: i) If
$y\neq x$, then $\lambda^*y(x):= \ck x$; ii) $\lambda^*y(y):=
\cs\ck\ck$; iii) if $p, q$ are polynomials, then: $\lambda^*y(p q):=
\cs(\lambda^*y(p)) (\lambda^*y(q))$. From the fact that $\langle
X\rangle$ contains $\ck, \cs$ and it is closed under applications, we
deduce the condition on the coefficients of $\lambda^*y(t)$.
\end{proof}

We use combinatory completeness to define some combinators that we
will use later.
\begin{definition}\label{defi:basiccomb}
  Let $A$ be an $\oca$, we can define the following combinators or
  combinatorial functions that are elements of $\Phi$, or functions
  with codomain and domain $\Phi$.
  \begin{align*} \comb &= \lambda^* x\lambda^* y\lambda^* z(x(yz)) &
    \comi &= \lambda^* x (x) & \comc &= \lambda^* x\lambda^*
    y\lambda^* z (zxy) & \quad \quad \quad \comw &= \lambda^* x\lambda^* y(xyy)\\
    \comt &= \lambda^* x\lambda^* y (x) & \comf &= \lambda^*
    x\lambda^* y (y) & \comp &= \lambda^* x \lambda^* y \lambda^* z
    (zxy)&\hspace*{-1cm}\compz &= \lambda^* x (x\,\comt)\,\,\,\,\,\,
    \compo =
    \lambda^* x (x\,\comf)\\ 
\end{align*}
\vspace*{-1.2cm}
\[\coma: \Phi \times \Phi \rightarrow \Phi \,\,\,\,\,\,\,
\coma(r,s)= \lambda^*x(\comp(rx)(sx)) \quad \comd:\Phi \rightarrow
\Phi \quad \comd(f)=\lambda^* x (f (\compz x) (\compo x)).\] 
\end{definition}
\begin{lemma}\label{lem:basicomb2} 
 If $A$ is an $\oca$, the above definitions ensure that:
\begin{align*}\comb abc \leq a(bc)\,;\,\comi
  a \leq a \,;\, \comc abc \leq acb\,;\,& \comw ab \leq
  abb\,;\,\compz(\comp ab)\leq a\,;\,\compo(\comp a b)\leq b;\\ rc
  \leq a,sc\leq b \imp \coma(r,s)c \leq &\comp ab\,\,;\,\,
  \comd(f) \ell\leq f (\compz
  \ell) (\compo \ell).\\
\end{align*}
 for all $a,b,c,\ell\in A$.
\end{lemma}
\subsection{Implicative ordered combinatory algebras}

\begin{definition} \label{defi:ioca}An implicative ordered combinatory
  algebra --a $\ioca$--, consists of an $\operatorname{inf}$--complete
  partially ordered set $(A,\leq)$ equipped with:
\begin{enumerate}
\item binary operations
\[\rapp:A\times A\to A,\qquad(a,b)\;\mapsto\; a b\] called
\emph{application}, monotone in both arguments, and
\[\mathrm{imp}:A^\op\times A\to A,\qquad (a,b)\;\mapsto\; a\to b\]
called \emph{implication}, antitonic in the first argument and monotone
in the second;
\item a subset $\Phi\subseteq A$ (called \emph{filter}) which is
closed under application;
\item distinguished elements $\coms,\comk,\come\in\Phi$
\end{enumerate} such that the following holds for all $a,b,c\in A$.
\begin{itemize}
\item[\rpaxk] $\comk a b\leq a$
\item[\rpaxs] $\coms a b c\leq a c (b c)$
\item[\rpaxha] $a\leq b\to c\quad\imp\quad a b\leq c$
\item[\rpaxe] $a b\leq c\quad\imp\quad \come a \leq b\to c$
\end{itemize}
\end{definition}  
\medskip

\subsection{Krivine ordered combinatory algebras}

\begin{definition} \label{defi:oca}A Krivine ordered combinatory
  algebra --a $\koca$--, consists of an $\ioca$ equipped with a
  distinguished element $\comc\in\Phi$ such that for all $a,b\in A$,
\[\emph{\rpaxc} \quad \comc\leq ((a\to b)\to a )\to a.\]
\end{definition}  
\medskip

Next we show that in the definition of a $\koca$ (and also of a
$\ioca$) some of its elements are superfluous and can be obtained from
the others. Here we present a minimal setup for the concept.

\begin{definition}\label{defi:roca} A quadruple
  $\mathcal Q=(A,\leq,\rightarrow,\Phi)$ where 
\begin{enumerate}
\item The relation $\leq$ is a partial order in $A$ with the property
that each $X \subseteq A$ has an infimum.
\item The map $\rightarrow: A \times A \rightarrow A$
--\emph{implication}-- is Antone in the first variable and
monotone in the second.
\item $\Phi \subset A$ --a \emph{filter}-- is a subset of $A$.
\end{enumerate} 
is said to be proper if $\Phi$ satisfies the following
conditions.

Define:
\begin{enumerate}
\item a map $\rapp : A \times A \rightarrow A$ --called the
  application-- given for
  $a,b \in A$ as: $$\rapp(a,b)=ab:=\operatorname{inf}\{c:a \leq (b
  \rightarrow c)\}.$$ 
\item $\ck:=\operatorname{inf}\{a \rightarrow (b \rightarrow a): a,b
\in A\}$.
\item $\cs:=\operatorname{inf}\{a \rightarrow (b \rightarrow (c
\rightarrow (ac)(bc))): a,b,c \in A\}$.
\item $\ce:=\operatorname{inf}\{a \rightarrow (b \rightarrow ab): a,b
\in A\}$.
\item $\cc:=\operatorname{inf}\{(((a \rightarrow b) \rightarrow
a)\rightarrow a): a,b \in A\}$.
\end{enumerate}
then
\begin{enumerate}
\item The set $\Phi$ is closed under the application: $\rapp(\Phi,\Phi)
\subset \Phi$.
\item The elements $\ck,\cs,\ce, \cc \in \Phi$.
\end{enumerate}
\end{definition}

\begin{theorem}\label{theo:allconnected}
  If $\mathcal Q= (A,\leq,\rightarrow,\Phi)$ is proper (Definition
  \ref{defi:roca}), then $\mathcal A_\mathcal
  Q=(A,\leq,\rightarrow,\Phi,\rapp,\ck,\cs,\ce,\cc)$ is a $\koca$.
\end{theorem}
\begin{proof}
  To prove the \emph{half adjunction property}, assume that for $a,b,c
  \in A$, $a \leq (b \rightarrow c)$ as $ab$ is the infimum of the
  elements $c$ with the above property, it is clear that in this
  situation $ab \leq c$ (condition (PA) of Definition \ref{defi:oca}:
  the ``half adjunction property''). The fact that the application
  $(a,b) \mapsto ab$ is monotone in both variables follows directly
  from the definition by using the monotony properties of the
  implication.  If $a,b,c \in A$ are such that $ab \leq c$, by
  definition we know that $\ce \leq a\rightarrow (b \rightarrow ab)$
  and then that $\operatorname{e} \leq a\rightarrow (b \rightarrow
  c)$. Applying the \emph{half adjunction property} ((PA) in the
  previous definition) we obtain that $\ce a \leq (b \rightarrow
  c)$. The satisfaction by $\ck,\cs$ of the required properties
  follows by a direct application of (PA) and the condition for $\cc$
  is evidently satisfied.
\end{proof}

\section{Indexed preorders and triposes}
\subsection{Preorders, meet semi-lattices and Heyting preorders}\label{sec-preorders}

\begin{definition} We denote by $\catord$ the category of preorders
  and monotone maps. A preorder $(D,\leq)$ is a set $D$ with a
  reflexive and transitive relation $\leq$. A \emph{monotone map}
  between the preorders $(D,\leq)$ and $(E,\leq)$ is a function
  $f:D\to E$ such that $d\leq d'$ implies $f(d)\leq f(d')$ for all
  $d,d'\in D$. If $d\leq d'$ and $d'\leq d$, we say that $d$ and $d'$
  are \emph{isomorphic}, and write $d\cong d'$.
\end{definition}

\begin{definition} Let $\cord$ and $\dord$ be two preorders.
\begin{enumerate}
\item For monotone maps $f,g:\cord\to\dord$, we define $f\leq
  g\defequi \forall d\in D\qdot f(d)\leq g(d)$ and say that $f$ and
  $g$ are \emph{isomorphic} (written $f\cong g$) if $f\leq g$ and
  $g\leq f$.
\item A monotone map $f:\cord\to\dord$ is called an
  \emph{equivalence}, if there exists a monotone map $g:\dord\to\cord$
  such that $g\circ f\cong \id_D$, and $f\circ g\cong\id_E$ and $g$ is
  called a \emph{weak inverse} of $f$. In this situation we say that
  $\cord$ and $\dord$ are \emph{equivalent} (written $\dord\simeq\eord$).
\item Given monotone
maps $f:\cord\to\dord$, $g:\dord\to\cord$, we say that `$f$ is left
adjoint to $g$', or `$g$ is right adjoint to $f$', and write $f\adj
g$, if $\id_C\leq g\circ f$ and $f\circ g\leq \id_D$. 
\end{enumerate}
\end{definition}
\begin{remark}\label{remark:equiv-ord} The following assertions are
  easy to prove. 
\begin{enumerate}
\item \label{item:equiv-ord}A monotone map $f:\cord\to\dord$ is
an equivalence if and only if it is \emph{order reflecting} and
\emph{essentially surjective}, i.e.\
\begin{enumerate}
\item $\forall c,c'\in D\qdot f(c)\leq f(c')\imp c\leq c'$, and
\item $\forall d\in D\;\exists c\in C\qdot f(c)\cong d$.
\end{enumerate}
\item \label{item:adj} Let $f:\cord\to\dord$, $g:\dord\to\cord$
be monotone maps between preorders.
\begin{enumerate}
\item $f$ is left adjoint to $g$, if and only if
$\forall c\in C,\, d\in D\qdot f(c)\leq d \equi c\leq g(d)$.

\item Adjoints are unique up to isomorphism, i.e.\ when $f\adj g$ and
$f\adj g'$, then $g\cong g'$ (and similarly for left adjoints).
\end{enumerate}
\end{enumerate}
\end{remark}

\begin{definition} A meet semi-lattice is a preorder $\dord$ equipped
  with a binary operation $\wedge$ and a distinguished element $\top$
  such that for all $a,b,c\in D$:
\begin{enumerate}
\item \label{item:meet1} $a\wedge b\leq a$; 
\item \label{item:meet2} $a\wedge b\leq b$; 
\item \label{item:meet3} $c\leq a\text{ and } c\leq b \imp c\leq
  a\wedge b$;
\item \label{item:top} $a\leq \top$. 
\end{enumerate}
\end{definition}
\begin{remark} If $\dord$ is a meet semi-lattice, then the function
$(d,d')\mapsto d\wedge d'$ is a monotone map of type $D\times D\to D$,
which is right adjoint to the diagonal map $\delta:D\to D\times D$,
$d\mapsto (d,d)$.
\end{remark}
\begin{definition} $\catslat$ is the category of meet semi-lattices,
and \emph{meet preserving monotone maps}, i.e.\ monotone maps
$f:\dord\to\eord$ such that
\begin{enumerate}
\item $f(d)\wedge f(d')\cong f(d\wedge d')$ for all $d,d'\in D$
\item $f(\top)\cong \top$.
\end{enumerate}
\end{definition}
\begin{definition} We define $\cathpo$, the category of Heyting
  preorders and morphisms.
\begin{enumerate}
\item A \emph{Heyting preorder} is a meet semi-lattice $(A,\leq)$ with a
  binary operation \\ $\to:A\times A\to A$ (called \emph{Heyting
    implication}) satisfying
\begin{equation}\label{eq:heyting-implication} a\wedge b\leq
c\qqtext{if and only if}a\leq b\to c\tag{HI}
\end{equation} for all $a,b,c\in A$.
\item A \emph{morphism of Heyting preorders} 
is a monotone map $f:(A,\leq)\to (B,\leq)$ such that
\begin{enumerate}
\item $f(\top)\cong\top$
\item $f(a\wedge b) \cong f(a)\wedge f(b)$
\item $f(a\to b)\cong f(a)\to f(b)$
\end{enumerate} for all $a,b\in A$.

\end{enumerate}
\end{definition}
\begin{remarks}
\begin{enumerate}
\item The term `Heyting preorder' is not standard, but it is the same
as a `posetal Cartesian closed category', or equivalently a preorder
whose poset reflection is a `Heyting
semi-lattice'~\cite[Part~A1.5]{elephant1}.
\item A Heyting preorder with finite joins is what is called a
\emph{Heyting prealgebra}, e.g.\
in~\cite{kn:vanOosbook}. The anti-symmetric version is the
well known concept of \emph{Heyting algebra}.
\item To interpret disjunction we also want joins in triposes, but we
  don't have to postulate disjunction (and neither $\exists$), since
  they can be encoded in terms of the other connectives in second
  order logic.
\item Also, we don't have to demand Heyting implication to be monotone
  -- it follows from the definition that it is antitonic in the first,
  and monotone in the second variable.
\end{enumerate}
\end{remarks}
\subsection{Preorders associated to $\mathcal{AKS}$s,
  $\oca$\hspace*{1pt}s and $\ioca$\hspace*{1pt}s}

\[\mathcal{AKS}\]

\begin{definition} Let $\mathcal K=(\Lambda,\Pi,\cdots)$ be an
  abstract Krivine structure. We define the relation $\sqsubseteq$ in
  $\power(\Pi)$ as follows:
\begin{equation}\label{eq:entailment} P,Q \in \power(\Pi)\,, \quad 
  P\sqsubseteq  Q\qdefequi
  \exists t\in \operatorname{QP}\; t\orth P\imp Q
\end{equation} for $P, Q \in \power(\Pi)$. An element $t
\in \Phi$ as above is said to be ``a realizer of the
relation'' $P\sqsubseteq Q$''.
\end{definition}
\begin{remark}\label{remark:equal}
  Notice that the relation above, could have been defined
  using the arrow $\Rightarrow_{\perp}$. Indeed, $t\orth
  P \imp Q$ if and only if $t\orth
  P \imp_{\perp}{} Q$ by Lemma \ref{lem:aks}, \eqref{lem:aks-0}.
\end{remark}
\begin{lemma}\label{lem:aks-ord} Let $\mathcal K$ be an abstract
  Krivine structure, then the relation $\sqsubseteq$ is a preorder on
  $\power(\Pi)$.
\end{lemma}
\begin{proof} The combinator $\cI$ is
  a realizer of $P\sqsubseteq P$ for any $P \in \power(\Pi)$, thus $\sqsubseteq$ is reflexive. For transitivity, assume
  that $P, Q,R \in \power(\Pi)$, and that
  $t,u\in\operatorname{QP}$ are realizers of $P\sqsubseteq Q$ and
  $ Q\sqsubseteq R$, respectively.  Then $\cB tu$ is a realizer of
  $P\sqsubseteq R$.
\end{proof}

\begin{lemma}\label{lem:eqpo} The canonical inclusion
$\power_{\perp}(\Pi) \eincl\power(\Pi)$ is an equivalence of 
preorders with respect to $\sqsubseteq$. 
\end{lemma}
\begin{proof} By Remark~\ref{remark:equiv-ord} it is suffices to show
that the inclusion is order reflecting and essentially surjective. Since
the order on $\pobot(\Pi)$ is defined as restriction of the order
on $\power(\Pi)$ the first assertion is clear.

To prove that the inclusion is essentially surjective, we show that
$P\sqsubseteq({}^\perp P)^\perp$ and
$({}^\perp P)^\perp\sqsubseteq P$ for all
$P \in \power(\Pi)$. This holds since
\[ \cI\orth {}^\perp(({}^\perp P)^\perp)\cdot P = {}^\perp P \cdot
P\qqtext{and} \cI\orth {}^\perp P\cdot ({}^\perp P)^\perp \] for all
$P\subs\Pi$. Both relations are realized by $\cI$ as follows
directly from Lemma \ref{lem:aks}, \eqref{lem:aks-i-real} applied 
  in the cases of $P$ and $({}^\perp P)^\perp$, respectively.
\end{proof}
\vspace*{-.5cm}
\[\oca\]
\begin{definition}\label{defi:squareord} Let $(\mathcal A,\Phi)$
  be a filtered $\oca$. We define:
 \begin{enumerate}
\item The relation
  $\sqsubseteq_\Phi$ in $A$ as follows:
\[a \sqsubseteq_\Phi b, \text{if and only if}\,\, \exists f \in \Phi:
f a \leq b.\]
\item A map $\wedge: A \times A \rightarrow A$ as $a \wedge b:= \comp ab$
  --see Definition \ref{defi:basiccomb}. 
\item An element $\top \in \Phi$ 
\end{enumerate}
\end{definition}

Usually we omit the subscript $\Phi$ in the notation of the relation
$\sqsubseteq_\Phi$ and write $a \sqsubseteq b$. An the element
$f$ as above is said to be ``a realizer of the relation $a \sqsubseteq
b$'' and write this assertion as $f \Vdash a \sqsubseteq b$.

We establish some properties that will be of later use.
\begin{lemma} \label{lem:meetoca}If $(\mathcal A,\Phi)$ is a filtered
  $\oca$ then in the notations of Definition \ref{defi:basiccomb} we
  have that:
\begin{enumerate}
\item \label{item:meet11} $\compz \Vdash a \wedge b \sqsubseteq a$
\item \label{item:meet12} $\compo \Vdash a \wedge b \sqsubseteq b$
\item \label{item:meet13} If $r \Vdash c \sqsubseteq a$ and $s \Vdash c
\sqsubseteq b$ then $\coma(r,s) \Vdash c \sqsubseteq a \wedge b$
\item \label{item:meet14} $\comk \comk \Vdash a \sqsubseteq \top$. 
\end{enumerate}

Hence, $(A,\wedge,\sqsubseteq)$ is a meet-semi-lattice.
\end{lemma}
\begin{proof}
All the assertions follow directly from Lemma \ref{lem:basicomb2}.
\end{proof}
\[\ioca\]
\medskip
Next we show that in the case of the existence of an adjunctor, more
precise assertions can be proved concerning the meet and the order
$\sqsubseteq$.
\begin{theorem}\label{theo:meetadj} If $(\mathcal A,\Phi)$ is
  a $\ioca$ then:
\begin{enumerate}
\item If $a,b \in A$ then $a \sqsubseteq b$ if and only if there is an
  element $f \in \Phi$ such that $f \leq a \rightarrow b$.
\item If $a,b,c \in A$: \[ a \wedge b \sqsubseteq c \Leftrightarrow a
  \sqsubseteq (b \rightarrow c).\] In other words $(A, \sqsubseteq,
  \wedge, \rightarrow)$ is a Heyting preorder.
\end{enumerate}
\end{theorem}
\begin{proof}
\begin{enumerate}
\item Assuming that $f \leq a \rightarrow b$ and using the half
  adjunction property we deduce that $fa \leq b$ i.e. that $a
  \sqsubseteq b$. In case that $a \sqsubseteq b$, first we deduce that
  $ga \leq b$ for some $g \in \Phi$. Using the adjunctor we deduce
  that $\ce g \leq a \rightarrow b$. 
\item To see that the map $\to$ gives a Heyting implication on
  $(A,\sqsubseteq)$, we have to check that
\[a \wedge b\sqsubseteq c \quad\equi\quad a\sqsubseteq (b\to c)\]
where $(a \wedge b) =\comp ab$.  

If the left inequality holds, there exists an element $f\in\Phi$ such
that $fa \leq b \to c$, and Definition \ref{defi:oca}, \rpaxha\ gives
$fab\leq c$.  In accordance with Lemma \ref{lem:basicomb2} there
exists a function $\comd:\Phi \rightarrow \Phi$ such that $\comd(f)
\ell\leq f (\compz \ell) (\compo \ell)$ for all $\ell \in A$, and this
gives (substituting $\ell$ by $\comp ab$)
\[\comd(f)(a \wedge b) = \comd(f)(\comp ab)\leq f ab \leq c.\] 

Conversely, assume that the right hand side holds, i.e.\ there
exists an $f\in\Phi$ such that $f
(\comp\,ab)\leq c$.  Then we can deduce
\[ \def\fCenter{\;\leq\;} \AX$f ( \comp\,ab)\fCenter c$ \UI$\comb
f(\comp\,a)\,b\fCenter c $ \UI$\comb (\comb f) \comp
\,ab\fCenter c$ \UI$\come (\comb (\comb f)
\comp \,a)\fCenter b\to c$ \UI$\comb \come
(\comb (\comb f) \comp) \,a\fCenter b \to c,$
\DP
\] hence $\comb \come (\comb (\comb f) \comp)$ is a realizer of
$a \sqsubseteq b\to c$.
\end{enumerate}
\end{proof}

For future use we prove the following property of the combinator
$\comc$ in the case that the $\ioca$ is equipped with one.
\begin{lemma} \label{lem:classical} Assume that the $\mathcal A$ is a
  $\ioca$ equipped with an element $\comc$ with the property that for
  if $a,b \in \mathcal A$ then, $\comc \leq ((a \to b) \to a) \to a$.
  If $a \in \mathcal A$, then $\comc \Vdash ((a \to \bot) \to \bot)
  \sqsubseteq a$.
\end{lemma}
\begin{proof} For $a \in \mathcal A$ we have:
  \begin{prooftree} 
\Axiom$\comc \leq  ((a \to \bot)\fCenter\ \to a)\to \fCenter\ a$
\UnaryInf$\comc((a \to \bot)\fCenter\ \to a)\leq  a$
\UnaryInf$\comc((a \to \bot) \to \bot) \leq \comc((\fCenter\ a \to
  \bot) \to a) \leq  a$
\UnaryInf$\comc \Vdash ((a \to \bot) \fCenter\ \to \bot)
  \sqsubseteq a$  
\end{prooftree}
\end{proof} 
\subsection{Indexed preorders and indexed
  meet-semi-lattices}\label{sec-indexed-posets} 

\begin{definition} \begin{enumerate}
\item An \emph{indexed preorder} is a functor
$\fifd:\setord.$ 
\item An \emph{indexed meet-semi-lattice} is a functor
$\fifa:\catset^\op\to\catslat$.
\item An \emph{indexed Heyting preorder} is a functor
$\trip:\catset^\op\to\cathpo$.
\end{enumerate}
\end{definition}
We only present the following definitions in the case of preorders,
in the case of indexed meet-semilattices the concepts are similar. 
\begin{remarks}
\begin{enumerate}
\item Indexed preorders (in particular triposes, defined below) can be
used as categorical models of predicate logic. With this in mind, we
often call their elements \emph{predicates} -- more precisely, if
$\fifd$ is an indexed preorder, $I$ is a set, and
$\varphi\in\fifd(I)$, we say that $\varphi$ is a \emph{predicate on
$I$}.
\item If $\fifd$ is an indexed preorder and $f:J\to I$ is a function,
applying the functor to $f$ gives us a monotone map
$\fifd(f):\fifd(I)\to\fifd(J)$. We call this function \emph{reindexing
along $f$}, and usually abbreviate it by $f^*$. Thus, if $\varphi$ is
a predicate on $I$, then \emph{its reindexing $f^*(\varphi)$ along
$f$} is a predicate on $J$. Semantically, reindexing corresponds to
\emph{substitution and context extension}.
\item{There are more general concepts \hyphenation{pseu-do-func-tor}
    of indexed preorder, one is that of a \emph{pseudofunctor} of type
    $\setord$.} Another generalization of indexed preorders is to
  replace $\catset$ by another category. We do not need these levels of
  generality.
\item Preorders are a special case of \emph{indexed categories},
  which are functors $\fibc:\catset^\op\to\catcat$. The link between
  indexed categories and logic was discovered by Lawvere in the
  60ies~\cite{lawvere1969adjointness,lawvere1970equality}
  (`quantifiers as adjoints'), and is at the heart of 
  \emph{categorical logic}.
\end{enumerate}
\end{remarks}
\begin{definition} Given indexed preorders $\fifd,\fife:\setord$, an
\emph{indexed monotone map} $\sigma:\fifd\to\fife$ is a family
\[ \sigma_I:\fifd(I)\to \fife(I) \qquad (I\in\catset)
\] of monotone functions, such that we have
\begin{equation}\label{eq:pseudo} \sigma_J(f^*(\varphi))\cong
f^*(\sigma_I(\varphi))
\end{equation} for all functions $f:J\to I$ and predicates
$\varphi\in\fifd(I)$.
\end{definition}
\begin{remarks}
\begin{enumerate}
\item Indexed monotone maps are special cases of \emph{pseudo-natural
transformations}~\cite{lack20102}. If we have equality
in ~\eqref{eq:pseudo}, we speak of a \emph{strict} indexed monotone
map, which is an instance of a \emph{2-natural transformation}.
\item Indexed preorders and indexed monotone maps form a category,
which we denote by $\catiord$. Composition of indexed monotone maps
$\fifc\xrightarrow{\sigma}\fifd\xrightarrow{\tau}\fife$ is defined by
$(\tau\circ\sigma)_I(\varphi)=\tau_I(\sigma_I(\varphi))$ for
$\varphi\in \fifc(I)$. The identity $\id_\fifd$ of an indexed preorder
$\fifd$ is defined by $\id_{\fifd,I}(\varphi)=\varphi$ for all
$\varphi\in\fifd(I)$.
\end{enumerate}
\end{remarks}
\begin{definition} Let $\fifd,\fife$ be indexed preorders.
\begin{enumerate}
\item For indexed monotone maps $\sigma,\tau:\fifd\to\fife$, we define
\[ \sigma\leq\tau\defequi \forall I\in \catset\qdot
\sigma_I\leq\tau_I.
\] We say that $\sigma$ and $\tau$ are \emph{isomorphic}, and write
$\sigma\cong \tau$, if $\sigma\leq \tau$ and $\tau\leq \sigma$.
\item An indexed monotone map $\sigma :\fifd\to\fife$ is called an
\emph{equivalence}, if there exists a indexed monotone map $\tau
:\fife\to\fifd$ such that $\tau \circ \sigma\cong \id_\fifd$, and $\sigma
\circ \tau\cong\id_\fife$. In this case, $\tau $ is called an
\emph{(indexed) weak inverse} of $\sigma $.
\item We say that $\fifd$ and $\fife$ are \emph{equivalent}, and write
$\fifd\simeq\fife$, if there exists an equivalence $\sigma
:\fifd\to\fife$.
\end{enumerate}
\end{definition}
\begin{lemma}\label{lem:equiv-iord} An indexed monotone map
$\sigma:\fifd\to\fife$ is an equivalence, if and only if for every set
$I$, the monotone map $\sigma_I:\fifd(I)\to\fife(I)$ is order
reflecting and essentially surjective.
\end{lemma}
\begin{proof} By Remark~\ref{remark:equiv-ord},~\eqref{item:equiv-ord}, every $\sigma_I$ has a
weak inverse $\tau_I:\fife(I)\to\fifd(I)$. Together these $\tau_I$
give rise to an indexed weak inverse of $\sigma$.
\end{proof}
\subsection{Triposes}\label{sec-triposes}
Next we consider a special kind of indexed Heyting preorders, called
\emph{triposes}, see \cite{kn:tripos}.
\begin{definition}\label{def:tripos} 
A \emph{tripos} is a
functor $ \trip:\setopto\cathpo $ such that
\begin{enumerate}
  \item For every function $f:J\to I$, the reindexing map
$f^*:\trip(I)\to\trip(J)$ has a right adjoint
$\forall_f:\trip(J)\to\trip(I)$.
\item If
\begin{equation}\label{eq:pullback} \vcenter{\xymatrix{
P\pullbackcorner \ar[r]^q \ar[d]_p & K \ar[d]^g \\ J \ar[r]_f & I }}
\end{equation} is a pullback square of sets and functions, then
$\forall_q(p^*(\varphi))\cong g^*(\forall_f(\varphi))$ for all
$\varphi\in\trip(J)$ (this is the \emph{Beck-Chevalley condition}).
\item\label{def:tripos_chi} $\trip$ has a \emph{generic predicate},
i.e.\ there exists a set $\prop$, and a $\triptr\in\trip(\prop)$ such
that for every set $I$ and $\varphi\in\trip(I)$ there exists a (not
necessarily unique) function $\chi_\varphi:I\to\prop$ with
$\varphi\cong\chi_\varphi^*(\triptr)$.  

\end{enumerate}
\end{definition}
\begin{remark}
\begin{enumerate}
\item 
$\forall_f:\trip(J)\to\trip(I)$ is not required
to preserve meets or implication.
\item 
The statement that the above
square is a pullback, means explicitly that
\[\forall j\in J,\, k\in K\qdot f(j)=g(k)\equi \Big(\exists !x\in P\qdot
p(x)=j\wedge q(x)= k\Big).\]
\end{enumerate}
\end{remark}
\begin{lemma}\label{lem:equiv-iord-trip} Let $\fifd$ and $\trip$ be
indexed preorders, and assume that $\sigma:\fifd\to\trip$ and
$\tau:\trip\to\fifd$ form an equivalence. If $\trip$ is a tripos, then
so is $\fifd$.
\end{lemma}
\begin{proof} This is because all the defining properties of a tripos
are stable under equivalence, and can be transported along $\sigma$
and $\tau$. In particular:
\begin{enumerate}
\item for any set $I$, $\tau_I(\top)$ is a greatest element in
$\fifd(I)$
\item meets in $\fifd(I)$ are given by
$\varphi\wedge\psi=\tau_I(\sigma_I(\varphi)\wedge\sigma_I(\psi))$
\item Heyting implication in $\fifd(I)$ is given by
$\varphi\to\psi\;=\;\tau_I(\sigma_I(\varphi)\to\sigma_I(\psi))$
\item universal quantification in $\fifd$ is can be defined by
$\forall_f(\varphi)=\tau_I(\forall_f(\sigma_J(\varphi)))$ for $f:J\to
I$ and $\varphi\in\fifd(J)$
\item a generic predicate for $\fifd$ is given by
$\tau_\prop(\triptr)$ where $\triptr\in\trip(\prop)$ is the generic
predicate of $\trip$
\end{enumerate}
\end{proof}
\section{Constructing triposes from ordered structures}
In this section we show how to construct triposes --or weaker structures
such as indexed meet-semilattices or indexed preorders-- from ordered
combinatory algebras or abstract Krivine structures. We also consider
the relations between the different constructions.
\subsection{From $\mathcal{AKS}$s to indexed preorders}

\begin{definition} Let $\mathcal K=(\Lambda,\Pi,\cdots)$ be an
  abstract Krivine structure, and let $I$ be any set. The
  \emph{entailment relation} $\ent$ in $\power(\Pi)^I$ is defined by
\begin{equation}\label{eq:entailment} \varphi,\psi \in
  \power(\Pi)^I\,,\quad \varphi\ent \psi\qdefequi
  \exists t\in \operatorname{QP}\;\forall i\in I\qdot t\orth \varphi(i)\imp\psi(i)
\end{equation} for $\varphi,\psi:I\to \power(\Pi)$. An element $t
\in \Phi$ as above is said to be ``a realizer of the
entailment $\varphi\ent\psi$''.
\end{definition}
\begin{remark}
  Notice that the entailment relation above, could have been defined
  using the arrow $\Rightarrow_{\perp}$, because $t\orth
  \varphi(i)\imp\psi(i)$ if and only if $t\orth
  \varphi(i)\imp_{\perp}{}\psi(i)$, compare with Remark
  \ref{remark:equal}.
\end{remark}
\begin{lemma}\label{lem:aks-iord} Let $\mathcal K$ be an abstract
Krivine structure.
\begin{enumerate}
\item \label{lem:aks-iord-iord} For any set $I$, the entailment
relation $\ent$ is a preorder on $\power(\Pi)^I$.
\item \label{lem:aks-iord-monot} For any function $f:J\to I$,
precomposition defines a monotone map
\[f^*:(\power(\Pi)^I,\ent)\to (\power(\Pi)^J,\ent),\qquad
\varphi\mapsto \varphi\circ f.\]
\item \label{lem:aks-iord-func} The preceding constructions give an
indexed preorder
\[ \trip(\mathcal K):\setord,\qquad I\mapsto
(\power(\Pi)^I,\ent),\qquad f\mapsto f^*.
\]
\end{enumerate}
\end{lemma}
\begin{proof} \ad{\ref{lem:oca-islat-slat}} This is proved in the same
  way as Lemma \ref{lem:aks-ord}. 

\medskip

\ad{\ref{lem:oca-islat-monot}} Let $\varphi,\psi:I\to
\power(\Pi)$. If $t\in\operatorname{\Phi}$ is a realizer of $\varphi\ent\psi$,
then it is also a realizer of $\varphi\circ f\ent \psi\circ f$, thus
$f^*$ is monotone.

\medskip \ad{\ref{lem:oca-islat-func}} We check the
functoriality condition, i.e.\ $g^*\circ f^*=(f\circ g)^*$ and
$\id_I^*=\id_{A^I}$ for $K\xrightarrow{g}J\xrightarrow{f}I$. This
follows from associativity and unit laws for composition.
\end{proof}

\begin{definition}\label{def:aks-iord2} The indexed preorder
$\trip_\bot(\mathcal K):\setord$ is defined by
\[ \trip_\bot(\mathcal K)(I) = (\pobot(\Pi)^I,\ent),\qquad f\mapsto f^*
\] where the order on $\pobot(\Pi)^I$ is the restriction of the
entailment order on $\power(\Pi)^I$ to predicates with values in
$\pobot(\Pi)$.
\end{definition}
\begin{lemma} The canonical inclusion $\trip_\bot(\mathcal
  K)\eincl\trip(\mathcal K)$ is an equivalence of indexed preorders.
\end{lemma}
\begin{proof} By Lemma~\ref{lem:equiv-iord} it is suffices to show
that the inclusion
\[(\pobot(\Pi)^I,\ent)\eincl(\power(\Pi)^I,\ent)\] is an
equivalence for all sets $I$ and this is proved in the same way as Lemma
\ref{lem:eqpo}.  
\end{proof}
\subsection{From $\oca$s to indexed meet-semilattices}

\begin{definition}\label{def:oca-to-islat} Let $(\mathcal A,\Phi)$ 
  be a filtered $\oca$. The
  \emph{entailment relation} $\ent~ \subs A^I\times A^I$ is defined by
\begin{equation}\label{eq:entailment} \varphi\ent \psi\qdefequi
\exists r\in \Phi\;\forall i\in I\qdot r(\varphi(i))\leq\psi(i)
\end{equation} for $\varphi,\psi:I\to A$. An $r \in \Phi$ as above is
said to be ``a realizer of the entailment
$\varphi\ent\psi$''.
\end{definition}
\begin{lemma}\label{lem:oca-islat} Let $(\mathcal A,\Phi)$ be
  a filtered $\oca$.
\begin{enumerate}
\item \label{lem:oca-islat-slat} For any set $I$, the entailment
relation $\ent$ is a preorder on $A^I$, and $(A^I,\ent)$ is a
meet-semi-lattice with the following definitions:
$\top:I\to A$; $\top(i)=\top=\comk$ and $\varphi\wedge\psi$ of two
functions $\varphi,\psi:I\to A$ is $(\varphi\wedge\psi)(i)= \varphi (i) \wedge \psi (i)$.
\item \label{lem:oca-islat-monot} For any function $f:J\to I$,
precomposition with $f$ defines a meet preserving monotone map
\[f^*:(A^I,\ent)\to (A^J,\ent),\qquad \varphi\mapsto \varphi\circ f.\]
\item \label{lem:oca-islat-func} The preceding constructions define an
indexed meet-semi-lattice
\[ \trip(\mathcal A):\catset^\op\to\catslat,\qquad I\mapsto
(A^I,\ent),\qquad f\mapsto f^*.
\]
\end{enumerate}
\end{lemma}
\begin{proof} \ad{\ref{lem:oca-islat-slat}}  We use in the proof of
  this assertion the realizers exhibited in Lemma \ref{lem:meetoca}.  
\medskip

\ad{\ref{lem:oca-islat-monot}} Let $\varphi,\psi:I\to A$. If
$r\in\Phi$ is a realizer of $\varphi\ent\psi$, then it is also a
realizer of $\varphi\circ f\ent \psi\circ f$, thus $f^*$ is
monotone. For meets, we have
\[ ((\varphi\wedge\psi)\circ f)(j) = \comp(\varphi(fj)\psi(fj)) =
((\varphi\circ f)\wedge(\psi\circ f))(j)
\] for all $j\in J$, which means
$f^*(\varphi\wedge\psi)=f^*(\varphi)\wedge f^*(\psi)$. Preservation of
$\top$ is shown in the same way.

\medskip \ad{\ref{lem:oca-islat-func}} It remains to check
functoriality, i.e.\ $g^*\circ f^*=(f\circ g)^*$ and
$\id_I^*=\id_{A^I}$ for $K\xrightarrow{g}J\xrightarrow{f}I$. This
follows from associativity and unit laws for composition.
\end{proof}
\subsection{From $\ioca$\hspace*{1pt}s to triposes}$ $

\medskip
Next we show that if the $\oca$ considered above has the necessary
additional structure to make it a $\ioca$, the indexed meet
semi-lattice just constructed is in fact a tripos. 

For any $\ioca$, $\mathcal
A=(A,\leq,\rapp,\rimp,\Phi,\comk,\coms,\come)$, the quintuple
$(A,\leq,\rapp,\comk,\coms)$ is an $\oca$ --that we also call
$\mathcal A$--, and $\Phi$ is a filter on it.  Thus, we can construct
the indexed meet-semi-lattice $\trip(\mathcal A)$ from
Definition~\ref{def:oca-to-islat}.

\begin{theorem} \label{theo:main} If $\mathcal
  A=(A,\leq,\rapp,\rimp,\Phi,\comk,\coms,\come)$ is a $\ioca$, then
  $\trip(\mathcal A)$ is a tripos.  Moreover, if the $\ioca$ is a
  $\koca$ --with combinator $\comc \in \Phi$--then $\neg \neg \varphi
  \vdash \varphi$, with $\neg \varphi:=\varphi \rightarrow \bot$.
\end{theorem}
\begin{proof} We know that $\trip(\mathcal A)$ is an indexed
meet-semi-lattice, and it remains to show that it has implication,
universal quantification, and a generic predicate.

For $\psi,\theta:I\to A$, we define $\varphi\to\psi$ by
\[ (\psi\to\theta)(i) = \varphi(i)\to\psi(i)
\] To see that this gives a Heyting implication on $(A^I,\ent)$, we
have to check that
\[\varphi\ent\psi\to\theta \quad\equi\quad
\varphi\wedge\psi\ent\theta \] where
$(\varphi\wedge\psi)(i)=\comp\varphi(i)\psi(i)$.  In a similar manner
as before, the assertion is shown in the same way as Theorem
\ref{theo:meetadj}.

Universal quantification of a predicate $\psi:J\to A$ along a function
$f:J\to I$ is defined by
\[ \forall_f(\psi)(i)=\inf_{f(j)=i}\psi(j)
\] With this definition it follows directly that for any
$\varphi:I\to A$ and $r\in\Phi$ we have

\begin{align*} 
& \forall j\in J\qdot r\,\varphi(f(j))\;\leq\;\psi(j)\\ 
\equi\quad & \forall i\in I\qdot r\,\varphi(i)\;\leq\;\forall_f(\psi)(i), 
\end{align*}
which means that $f^*\varphi\ent\psi$ if and only if
$\varphi\ent\forall_f\psi$ (with the same realizer), hence
Remark \ref{remark:equiv-ord},~\eqref{item:adj} implies that $\forall_f$ is right adjoint to
$f^*$.

For the Beck-Chevalley condition, consider the pullback
square~\eqref{eq:pullback} in Definition~\ref{def:tripos}, and let
$\varphi:J\to I$. For $k\in K$ we have
\begin{align*} g^*(\forall_f(\varphi))(k) & = \inf_{fj=gk}\varphi(j)\\
\text{and}\quad \forall_q(p^*(\varphi))(k) &= \inf_{\substack{x\in
P\\qx = k}} \varphi(p(x)).
\end{align*} In the first case, the infimum is taken over the set
$\setof{j\in J}{f(j)=g(k)}$, and in the second case over the set
$\setof{j\in J}{\exists x\in P\qdot p(x) = j,q(x)=k}$. These two sets
are equal since the square is a pullback (thus the Beck Chevalley
condition holds even up to equality).

Finally, a generic predicate for $\trip(\mathcal A)$ is given by
$\id_A\in\trip(\mathcal A)(A)$.

The fact that $\neg \neg \varphi \vdash \varphi$ for
all predicates $\varphi$, follows directly from Lemma
\ref{lem:classical}.

\end{proof}

\subsection{From $\mathcal {AKS}$s to $\koca$s}\label{subsection:akstooca}
$ $ \medskip

 Next, we recall the construction due to Streicher (see
 \cite{kn:streicher}) that starting from an $\mathcal
 {AKS}$ abbreviated as $\mathcal K$ produces a  $\koca$ --that we call
 $\mathcal A_\mathcal K$-- and show that they induce isomorphic
 indexed preorders --in fact triposes--. 

\begin{definition}\label{defi:paksoca} Given an $\mathcal {AKS}$:
  \[\mathcal
  K=(\Lambda,\Pi,\Perp,\operatorname{push},\operatorname{app},
  \operatorname{store}, \cK,\cS,\cCC,\operatorname{QP})\] define
  \[\mathcal A_\mathcal
  K=(A,\leq,\rapp,\rimp,\comk,\coms,\comc,\come,\Phi)\] as follows.
\begin{enumerate}
\item $(A,\leq)=(\pobot(\Pi),\sups)$;
\medskip
\item $\rapp(P,Q)=P \circ_{\perp} Q=({}^\perp({}^\perp Q \leadsto
  P))^\perp$, $\mathrm{imp}(P,Q)=P \Rightarrow_{\perp} Q=
  ({}^\perp({}^\perp P \push Q))^\perp$;
\medskip
\item $\ck=\{\cK\}^\perp$, $\cs=\{\cS\}^\perp$,
  $\cc=\{\cCC\}^\perp$, $\ce=\{\cE\cE\}^\perp$ , where
  $\cE=\cS(\cK(\cS\cK\cK))$;
\medskip
\item $\Phi=\setof{P\in\pobot(\Pi)}{\exists t\in \kkqp\qdot
t\perp P}$.
\end{enumerate} 
If $a,b \in A$ we write $ab:=\rapp(a,b)$ and $a \to b:=\rimp(a,b)$.
See Definitions \ref{defi:maps2}, \ref{defi:aks} and \ref{defi:maps1}.  
\end{definition} 

We recall the following important theorem from \cite{kn:streicher} and
write down a short proof for later use. 

\medskip
\begin{theorem}\label{theo:pakstooca} Let $\mathcal K$ be 
  an $\mathcal{AKS}$ and consider the structure $\mathcal
  A_\mathcal K$ presented in Definition \ref{defi:paksoca}.
\medskip
\begin{enumerate}
\item  Then, $\mathcal A_\mathcal K$ is a
  $\koca$. 
\medskip
\item The associated indexed preorders $\trip_\bot(\mathcal K)$ and
  $\trip(\mathcal A_\mathcal K)$ are isomorphic.
\medskip
\end{enumerate}
\end{theorem}
\begin{proof} 
\begin{enumerate}
\item The order is clearly inf complete as we observed in Remark
  \ref{remark:initialrl}. The fact that the implication and
  application satisfy the monotonicity properties, is clear. The
  implication $\rightarrow$ satisfies the \emph{half adjunction
    property}: if $a \leq (b \rightarrow c)$ then $ab \leq c$ as was
  established in Theorem \ref{theorem:adjunction}.

Next we prove that $\ck ab \leq a$.  Lemma \ref{lem:aks} (2)
guarantees that for all $a,b \in A$, $\cK \in {}^{\perp}\big({}^\perp
a.({}^\perp b . a)\big)$.  This assertion means that $\{\cK\}
\subseteq {}^{\perp}\big({}^\perp a.({}^\perp b . a)\big)$ and then
$\ck \supseteq \Big({}^{\perp}\big({}^\perp a.({}^\perp b
. a)\big)\Big)^\perp \supseteq {}^\perp a.({}^\perp b . a)$ that can
be written as ${}^\perp a \leadsto \ck \supseteq {}^\perp b
. a$. Moreover, from Definition \ref{defi:maps2}, \eqref{eqn:maps1} we
deduce that $\ck \circ_{\perp}\, a \supseteq {}^\perp a \leadsto \ck \supseteq
{}^\perp b . a$, i.e. $\ck \circ_{\perp}\,a \leq (b \rightarrow a)$ --compare
with Definition \ref{defi:maps2}--. Using the half adjunction property
(Theorem \ref{theorem:adjunction}), we deduce that $\ck ab \leq a$.

The condition $\cs abc \leq (ac)(bc)$ can be proved as follows.  Take
$t \perp a$, $s \perp b$, $u \perp c$ then using Lemma
\ref{lem:circdiamond} we deduce that $(su) \perp bc$ and $(tu) \perp
ac$ and also that  $(tu)(su) \perp (ac)(bc)$ and if $\pi \in (ac)(bc)$
is an arbitrary element we conclude that $(tu)(su) \perp \pi$.  
Then by the Definition \ref{defi:aks}, (S3) we conclude that $\cS
\perp t.s.u.\pi$. Hence we have proved that
$\cS \perp {}^\perp a . {}^\perp b .{}^\perp c .(ac)(bc)$ or $\cS \in
{}^\perp({}^\perp a . {}^\perp b .{}^\perp c .(ac)(bc))$ or
equivalently that $\cs \supseteq
{}^\perp a . {}^\perp b .{}^\perp c .(ac)(bc)$.

Assume now that we have a situation as follows: $x,y \in A$, $z
\subseteq \Pi$ with ${}^\perp x \push z \subseteq y$, clearly it
follows from Remark \ref{remark:firstadj} that $z \subseteq
{}^\perp x \leadsto y$

If we apply repeatedly the above observation to $
{}^\perp a . {}^\perp b .{}^\perp c .(ac)(bc) \subseteq \cs$ we deduce
that $(ac)(bc) \subseteq \cs a b c$, and the proof of this part is
finished.

The proof that $\ce$ as introduced in Definition \ref{defi:paksoca},
is an adjunctor is the content of Theorem
\ref{theo:adjunctionconverse}.

The proof that $\Phi \subseteq A$ is a filter that contains
$\ck,\cs,\ce$ is the following. The subset $\Phi$ is closed under
application because if
$f,g \in \Phi$, i.e. if we have $t_f\in {}^\perp f \cap
\operatorname{QP}$ and $t_g\in {}^\perp g \cap \operatorname{QP}$ then
$t_ft_g\in {}^\perp f {}^\perp g \cap \operatorname{QP} \subseteq
{}^\perp(f \circ_\perp g) \cap \operatorname{QP}$ (Lemma
\ref{lem:circdiamond}).  Moreover, $\ck,\cs, \ce \in \Phi$ because $\cK
\in {}^\perp \ck \cap \operatorname{QP}$, $\cS \in {}^\perp \cs \cap
\operatorname{QP}$ and $\cE \cE \in {}^\perp\ce \cap
\operatorname{QP}$.

Finally, as we took $\cc=\{\cCC\}^\perp$, it is clear that: $\cCC \in
{}^\perp\cc \cap \operatorname{QP}$. Moreover, we proved in Lemma
\ref{lem:aks} that $\cCC \in
{}^\perp(((a \rightarrow b) \rightarrow a)\rightarrow a)$, that
implies that $\operatorname c \supseteq ({}^\perp(((a \rightarrow b)
\rightarrow a)\rightarrow a))^\perp=(((a\rightarrow b)\rightarrow
a)\rightarrow a)$, i.e. $\operatorname c \leq (((a\rightarrow
b)\rightarrow a)\rightarrow a)$.
\item 
In both cases the predicates on a set $I$ are functions
$\varphi,\psi:I\to \pobot(\Pi)$, so we only have to check that the
two definitions of entailment coincide. The entailment in
$\trip(\mathcal A_\mathcal K)$ is
given by
\begin{align*} & \exists P\in \Phi\;\forall i\in I\qdot
P\varphi(i)\leq\psi(i)\\ \intertext{which using the adjunctor and
  substituting $P$ by $\ce P$ can be
formulated equivalently as:} & \exists P\in \Phi\;\forall i\in I\qdot
P\leq\varphi(i)\to\psi(i)\\
\end{align*}
As to the equivalence we have that:
\begin{align*} & \exists P\in \Phi\;\forall i\in I\qdot
P\leq\varphi(i)\to\psi(i)\\ \equi\quad & \exists t\in
\operatorname{QP}\;\forall i\in I\qdot
\{t\}^\perp\sups({}^\perp({}^\perp\varphi(i)\push\psi(i)))^\perp\\
\equi\quad & \exists t\in \operatorname{QP}\;\forall i\in I\qdot
t\perp ({}^\perp({}^\perp\varphi(i)\push\psi(i)))^\perp\\ \equi\quad & \exists t\in
\operatorname{QP}\;\forall i\in I\qdot
t\perp{}^\perp\varphi(i)\push\psi(i) \\
\equi\quad & \exists t\in
\operatorname{QP}\;\forall i\in I\qdot
t\perp{} \varphi(i)\Rightarrow \psi(i)\\
\end{align*} and the last line is the definition of entailment in
$\trip_\perp(\mathcal K)$.
\end{enumerate}
\end{proof}

\subsection{From $\koca$\hspace*{1pt}s to $\mathcal{AKS}$s}$ $ 

\medskip
In order to complete our program to set up the foundations of
realizability in terms of $\koca$\hspace*{1pt}s, we reverse the
construction presented in Subsection \ref{subsection:akstooca} and
show how to construct from  an $\koca$ called $\mathcal A$, an $\mathcal
{AKS}$ named as $\mathcal K_\mathcal A$. Then, we prove that the
corresponding triposes are equivalent.

\begin{definition}\label{def:aks-from-rdlp} Given a $\koca$ 
\[\mathcal A=(A,\leq,\rapp_{\mathcal A},\rimp,\comk,\coms,\comc,\come,\Phi)\]
we define the structure:
\[\mathcal
  K_\mathcal A=(\Lambda,\Pi,\Perp,\operatorname{push},\operatorname{app},
  \operatorname{store}, \cK,\cS,\cCC,\operatorname{QP})\]
as follows. 
\begin{enumerate}
\item $\Lambda=\Pi:=A$
\item $\Perp:=\leq$, i.e.\quad $s\orth\pi\defequi s\leq\pi$
\item $\mathrm{push}(s,\pi):=\rimp(s,\pi)=
  s\to\pi,\qquad \rapp(s,t):=\rapp_{\mathcal A}(s,t)=st,
  \qquad\mathrm{store}(\pi):=\neg\pi$
\item $\cK := \come(\comb\come\comk),\quad \cS :=
\come(\comb(\comb\come(\comb\come))\coms),\quad \cCC := \come\comc$
\item $\operatorname{QP}:=\Phi$
\end{enumerate}
\end{definition} Here, $\comb$ is an abbreviation for
$\coms(\comk\coms)\comk$, which has the property that $\comb a b c\leq
a(bc)$ for all $a,b,c\in A$, and $\neg \pi$ is a shorthand for
$\pi\to\bot$ and $\bot:=\operatorname{inf}(A)$.
\begin{theorem} 
In the notations of Definition \ref{def:aks-from-rdlp}
  the structure $\mathcal K_\mathcal A$ is an $\mathcal {AKS}$.

\end{theorem}
\begin{proof} 
It is clear that $\operatorname{QP}$ is closed under application
  and contains $\cK,\cS, \cCC$, and it remains to check the axioms
  about the orthogonality relation (see Definition
  \ref{defi:aks}). Substituting the above definitions, these axioms
  become:

\noindent $
\begin{array}{l@{\hspace{.5cm}}r@{\;\leq\;}l@{\quad\imp\quad}
r@{\;\leq\;}l} \text{(S1)} &t & u\to\pi& tu&\pi \\
\text{(S2)} &t&\pi &\come(\comb\come\comk)& t\to u\to\pi \\
\text{(S3)} & tv(uv)&\pi &
\come(\comb(\comb\come(\comb\come))\coms) & t\to u\to v\to\pi\\
\text{(S4)} & t&\neg\pi\to\pi&\come\comc& t\to\pi \\
\text{(S5)} &t&\pi&\neg\pi& t\to\pi',\quad \forall \pi'
\end{array} $

(S1)\ follows from Definition \ref{defi:oca}, (PA), and (S5)\ follows from
monotonicity of the arrow in its second argument and the antitonicity
in the first.

(S2)\ is shown by the following derivation.
\[ \def\fCenter{\; \leq\; } \AX$t\fCenter\pi$ \UI$\comk tu\fCenter
\pi$ \UI$\come(\comk t)\fCenter u\to\pi$ \UI$\comb\come\comk t\fCenter
u\to\pi$ \UI$\come(\comb\come\comk) \fCenter t\to u\to\pi$ \DP
\]

(S3)\ is proved using repeatedly the basic properties
of $\cb$ and $\ce$ as follows:
\[ \def\fCenter{\; \leq\; } \AX$tv(uv)\fCenter\pi$ \UI$\coms
tuv\fCenter\pi$ \UI$\come(\coms tu)\fCenter v\to\pi$
\UI$\comb\come(\coms t)u\fCenter v\to\pi$ \UI$ \come (\comb\come
(\coms t))\fCenter u\to v\to\pi$ \UI$\comb \come (\comb\come) (\coms
t)\fCenter u\to v\to\pi$ \UI$\comb(\comb\come(\comb\come)) \coms
t\fCenter u\to v\to\pi$ \UI$\come(\comb(\comb\come(\comb\come)) \coms)
\fCenter t\to u\to v\to\pi$ \DP
\]

Finally, $(S4)$\ is proved using the basic property of $\cC$
--Definition \ref{defi:oca}, \rpaxc, the monotony of the application
and the definition of $\ce$-- as follows:
\[ \AXC{\rpaxc} \UIC{$\comc\leq (\neg\pi\to\pi)\to\pi$}
\AXC{$t\leq\neg\pi\to\pi$} \BIC{$\comc t\leq
((\neg\pi\to\pi)\to\pi)(\neg\pi\to\pi)$} \AXC{$$}
\UIC{$(\neg\pi\to\pi)\to\pi\leq(\neg\pi\to\pi)\to\pi$}
\UIC{$((\neg\pi\to\pi)\to\pi)(\neg\pi\to\pi)\leq\pi$} \BIC{$\comc
t\leq\pi$} \UIC{$\come\comc\leq t\to\pi$} \DP
\]
\end{proof}

\begin{definition}
Let $(D,\leq)$ be a preorder. 
\begin{enumerate}
\item 
A \emph{principal filter} in $D$ is a subset of 
$D$ of the form 
\[
\aucl d_0:=\setof{d\in D}{d_0\leq d}.
\]
for some $d_0\in D$.
\item
Dually, a \emph{principal ideal} in $D$ is a subset of the form 
\[
\adcl d_0:=\setof{d\in D}{d\leq d_0}.
\]
for $d_0\in D$.
\end{enumerate}
\end{definition}
\begin{lemma}\label{lem:galois}
Let $\mathcal A$ be a $\koca$  structure, and $\mathcal K_\mathcal A$ the $\mathcal{AKS}$ induced via the
construction 
in Definition~\ref{def:aks-from-rdlp}. 
\begin{enumerate}
\item \label{lem:galois-u}
For $U\subs A$ we have ${}^\perp U =\adcl(\inf U)$, and $U^\perp=\aucl(\sup
U)$. 
\item \label{lem:galois-d}
For $a\in A$ we have $\inf(\aucl a) = a = \sup(\adcl a)$
\item\label{lem:galois-pbot} 
The set $\power_\bot(\Pi)$ consists precisely of the 
{principal filters} in $A$, and the maps
\[
f:A\to\pobot(\Pi),\, a\mapsto \aucl a
\qtext{and} g: \pobot(\Pi)\to A,\, P\mapsto \inf P,
\]
are mutually inverse and establish a bijection between $A$ and
$\pobot(\Pi)$.
\item\label{lem:galois-infimp} For $P,Q\in\power_\bot(\Pi)$ we have
  $\inf(P \imp_{\bot} Q)=\inf(P\imp Q)=\inf P\to \inf Q$.
\end{enumerate}
\end{lemma}
\begin{proof}
\ad{\ref{lem:galois-u}}
${}^\perp U$ is the set of lower bounds of $U$, and $\inf U$ is the
greatest lower
bound. An element $a\in A$ is a lower bound of $U$ if and only if it
is smaller
than the greatest lower bound. The second claim is just the dual
(recall that this duality is valid in a lattice). 

\ad{\ref{lem:galois-d}}
$a$ is a lower bound of $\aucl a$, and since $a\in\aucl a$ any other
lower bound
must be smaller. Thus $a$ is the greatest lower bound. The second part
is
symmetric.

\ad{\ref{lem:galois-pbot}} For $P\subs A$ we have 
$({}^\bot P)^\bot= (\adcl(\inf P))^\bot=\aucl(\inf P)$, thus all
$({}^\perp(-))^\perp$-stable
sets are principal filters.

Conversely, for a principal filter of the form $\aucl a$ and using the
previous parts of this Lemma, we have that 
$({}^\perp \aucl
a)^\perp=(\adcl(\operatorname{inf}(\aucl
a)))^\perp=\aucl(\operatorname{sup}(\adcl(\operatorname{inf}(\aucl
a))))=\aucl(\operatorname{sup}(\adcl a))=\aucl a$. 

To see that $f$ and $g$ are mutually inverse, take first $a\in A$. 
Then $g(f(a)) = \inf(\aucl a) = a$.
In the other direction, let $P\in \pobot(\kkpi)$. We know that $P$ is
a 
principal filter, thus $P=\aucl a$ for some $a\in A$ and we have
$f(g(P)) = \aucl(\inf P) = \aucl(\inf(\aucl a)) = \aucl a = P$.

\ad{\ref{lem:galois-infimp}} The fact that $\inf(P \imp_{\bot}
Q)=\inf(P\imp Q)$ follows also from the previous results. Indeed, we
have that $\operatorname{inf}(P \imp_{\perp}
Q)=\operatorname{inf}\big(({}^\perp(P \imp
Q))^\perp\big)=\operatorname{inf}((\adcl( \operatorname{inf}(P \imp
Q)))^\perp)=\operatorname{inf}(\adcl(\operatorname{inf}P \rightarrow
\operatorname{inf}Q)^\perp)= \operatorname{inf}((\adcl (a \rightarrow
b))^\perp)=\operatorname{\inf}(\aucl(\operatorname{sup}(\adcl(a
\rightarrow b))))=a \rightarrow b = \operatorname{inf} P \rightarrow
\operatorname{inf} Q= \operatorname{inf}(P \imp Q)$. In the above
computations we used that: $P= \aucl a,\, Q= \aucl b$ and the parts
(1), (2) and (3) already proved. The last equality is proved below. 

From the
preceding claim we know that given $P,Q$ as above, there are elements
$a,b \in A$ such that $P=\aucl a$ and $Q=\aucl b$. We have
\[
\aucl a\imp \aucl b \;=\; ^\perp(\aucl a)\push\aucl b 
\;=\; \adcl(\inf(\aucl a))\push\aucl b
\;=\; \adcl a\push\aucl b
\;=\; \setof{c\to d}{c\leq a, b\leq d}
\]
and thus $\inf(\aucl a\imp\aucl b)\; =\; a\to b$ by monotonicity of
the arrow.
\end{proof}

\begin{theorem}
The associated indexed triposes $\trip(\mathcal A)$ and
  $\trip_\bot(\mathcal K_\mathcal A)$ are equivalent (see Definitions
  ~\ref{lem:oca-islat}-\ref{lem:oca-islat-func} and
  in~\ref{lem:aks-iord}-\ref{lem:aks-iord-func} respectively).
\end{theorem}
\begin{proof}
Let $I$ be a set. The elements of $\trip(\mathcal A)(I)$ are
functions $\varphi:I\to A$, and the elements of
$\trip_\bot(\mathcal K_\mathcal A)(I)$ are functions $\widehat{\varphi}:I\to\pobot(\Pi)$.

 Post-composition with $f$ and $g$ from 
Lemma~\ref{lem:galois}-\ref{lem:galois-pbot} induces a bijection between
$\trip(\mathcal A)(I)$ and $\trip_\bot(\mathcal K_\mathcal A)(I)$, and
it remains to show that this bijection is compatible with the
entailment orderings.

Let $\varphi,\psi:I\to A$ be two predicates in $\trip(\mathcal A)(I)$,
with corresponding predicates $f\circ\varphi,f\circ\psi$ in
$\trip_\bot(\mathcal K_\mathcal A)(I)$. Then we can reformulate the
entailment $f\circ\varphi\ent f\circ\psi$ in $\trip_\bot(\mathcal
K_\mathcal A)(I)$ as follows:
\begin{align*} f\circ\varphi\ent f\circ\psi &\equi \exists
a\in\Phi\;\forall i\in I\qdot a\perp \aucl(\varphi i)\imp \aucl(\psi
i)\\ &\equi \exists a\in\Phi\;\forall i\in I\;\forall b\in
[\aucl(\varphi i)\imp \aucl(\psi i)] \quad \text{then} \quad a\leq b\\ &\equi \exists
a\in\Phi\;\forall i\in I \qdot a\leq \inf[\aucl(\varphi i)\imp
\aucl(\psi i)]\\ &\equi \exists a\in\Phi\;\forall i\in I \qdot a\leq
\varphi(i)\to\psi(i),
\end{align*} and this is equivalent to the entailment
$\varphi\ent\psi$ in $\trip(\mathcal A)(I)$:
\[ \exists a\in\Phi\;\forall i\in I \qdot a\varphi(i)\leq\psi(i)
\] by axioms \rpaxha\ and \rpaxe\ in Definition \ref{defi:oca}.

\end{proof}

\section{Internal realizability in
  $\koca$\hspace*{1pt}s}\label{section:nine} We
have shown that the class of ordered combinatory algebras that,
besides a filter of distinguished truth values are equipped with an
implication, an adjunctor and satisfy a completeness condition with
respect to the infimum over arbitrary subsets -- i.e.: $\koca$ s-- is
rich enough as to allow the Tripos construction and as such its
objects can be taken as the basis of the categorical perspective on
classical realizability --\`a la Streicher--. In this section we show
that we can define realizability in this type of combinatory algebras,
and thus, to define realizability in higher-order arithmetic.

\begin{definition}\label{defi:LomegaLang} Consider a set of constants of
kinds, one of its elements is denoted by $o$. The language of kinds is
given by the following grammar:
\[ \sigma, \tau :: = c\quad |\quad \sigma \to \tau
\] Consider an infinite set of variables labeled by kinds
$x^\tau$. Suppose that we have infinitely many variables labeled of
the kind $\tau$ for each kind $\tau$. Consider also a set of constants
$a^\tau, b^\sigma, \dots$ labeled with a kind. The language
$\mathcal{L}^\omega$ of order $\omega$ is defined by the following
grammar:
\[ M^\sigma, N^{\sigma\to\tau}, A^o, B^o :: = x^\sigma\quad |\quad
a^\sigma\quad |\quad (\lambda x^\sigma.M^\tau)^{\sigma\to\tau}\quad
|\quad (N^{\sigma\to\tau} M^\sigma)^\tau\quad |\quad (A^o\Rightarrow
B^o)^o\quad |\quad (\forall x^\tau.A^o)^o
\] $o$ represents the type of truth values. The expressions labeled
by $o$ are called ``formul\ae''. The symbols $\to$ and $\Rightarrow$,
when iterated, are associated on the right side. On the other hand,
the application, when iterated, are associated on the left side.
\end{definition}
\begin{definition}\label{defi:LomegaTyp} Let $A$ be a 
  $\koca$ and consider a set of variables $\mathcal{V}=\{x_1, x_2,
  \dots\}$. A declaration is a string of the shape $x_i:A^o$.  A
  context is a string of the shape $x_1:A_1^o, \dots, x_k:A_k^o$,
  i.e.: contexts are finite sequences of declarations. The contexts
  will be often denoted by capital Greek letters: $\Delta, \Gamma,
  \Sigma$. A sequent is a string of the shape $x_1:A_1^o, \dots,
  x_k:A_k^o\vdash p:B^o$ where $p$ is a polynomial of $A[x_1, \dots,
  x_k]$. The left side of a sequent is a context. When we do not
  make the declarations of the context of a sequent explicit, we will
  write it as~$\Gamma\vdash p:B^o$. Typing rules are trees with leaves
  of the shape
\begin{prooftree} \AxiomC{$S_1$} \AxiomC{$\dots$} \AxiomC{$S_h$}
\RightLabel{\scriptsize(Rule)} \TrinaryInfC{$S_{h{+}1}$}
\end{prooftree} where $h\geq 0$ and $S_1, \dots, S_{h{+}1}$ are
sequents. The typing rules for $\mathcal{L}^\omega$ are the following:
\begin{flushright}
  \begin{prooftree} \AxiomC{} \RightLabel{\scriptsize(ax)}
\LeftLabel{\scriptsize{(where $x_i:A^o_i$ appears in $\Gamma$)}}
\UnaryInfC{$\Gamma\vdash x_i:A^o_i$}
  \end{prooftree}
  \begin{prooftree} \AxiomC{$\Gamma, x:A^o\vdash p:B^o$}
\RightLabel{\scriptsize{$(\to_i)$}} \UnaryInfC{$\Gamma\vdash
\come(\lambda^*x\ p):(A^o\Rightarrow B^o)^o$}
  \end{prooftree}
  \begin{prooftree} \AxiomC{$\Gamma\vdash p:(A^o\Rightarrow B^o)^o$}
\AxiomC{$\Gamma\vdash q:A^o$} \RightLabel{\scriptsize{$(\to_e)$}}
\BinaryInfC{$\Gamma\vdash pq:B^o$}
  \end{prooftree}
  \begin{prooftree} \AxiomC{$\Gamma\vdash p:A^o$}
\RightLabel{\scriptsize{$(\forall_i)$}} \LeftLabel{\scriptsize{(where
$x^\sigma$ does not appears free in $\Gamma$)}}
\UnaryInfC{$\Gamma\vdash p:(\forall x^\sigma A^o)^o$}
  \end{prooftree}
  \begin{prooftree} \AxiomC{$\Gamma\vdash p:(\forall x^\sigma A^o)^o$}
\RightLabel{\scriptsize{$(\forall_e)$}} \UnaryInfC{$\Gamma\vdash
p:(A^o\{x^\sigma:=M^\sigma\})$}
  \end{prooftree}
\end{flushright}
\end{definition}
\begin{definition}\label{defi:LomegaSem} Let us consider
  $\mathcal{A}=(A,\leq, \rapp,  \to, \Phi, \ck, \cs, \ce, \cc)$ a
  $\mathcal{^KOCA}$. We define the interpretation of
  $\mathcal{L}^\omega$ as follows:
  \begin{enumerate}
    \item For \emph{kinds}: The interpretation of a constant $c$ is a
set $\llbracket c \rrbracket$. In particular, the constant $o$ is
interpreted as the underlying set of $\mathcal{A}$, i.e.: $\llbracket
o\rrbracket = A$. Given two kinds $\sigma, \tau$, the interpretation
$\llbracket \sigma \to \tau\rrbracket$ is the set of functions
$\llbracket \tau\rrbracket^{\llbracket \sigma \rrbracket}$

\item For \emph{expressions}: In order to interpret expressions,
we start choosing an assignment $\mathfrak{a}$ for the variables
$x^\sigma$ such that $\mathfrak{a}(x^\sigma)\in\llbracket
\sigma\rrbracket$. As it is usual in semantics, the substitution-like
notation $\{x^\sigma:=s\}$ affecting an assignment $\mathfrak{a}$
--or an interpretation using $\mathfrak{a}$--, modifies it by
redefining $\mathfrak{a}$ over $x^\sigma$ as the statement
$\mathfrak{a}\{x^\sigma:=s\}(x^\sigma):= s$. We proceed similarly for
interpretations.
      \begin{itemize}
      \item For an expression of the shape $x^\sigma$, its
interpretation is $\llbracket x^\sigma\rrbracket =
\mathfrak{a}(x^\sigma)$.
      \item For an expression of the shape $\lambda x^\sigma M^\tau$,
its interpretation is the function $\llbracket \lambda x^\sigma
M^\tau\rrbracket\in\llbracket \sigma \to \tau\rrbracket$ defined as
$\llbracket \lambda x^\sigma M^\tau\rrbracket(s):= \llbracket
M^\tau\rrbracket\{x^\sigma:=s\}$ for all $s\in\llbracket \sigma
\rrbracket$.
      \item For an expression of the shape $(N^{\sigma\to
\tau}M^\sigma)^\tau$ its interpretation is $\llbracket (N^{\sigma\to
\tau}M^\sigma)^\tau\rrbracket:=\llbracket N^{\sigma\to \tau}\rrbracket
\big(\llbracket M^\sigma \rrbracket\big) $
      \item For an expression of the shape $(A^o\Rightarrow B^o)^o$
its interpretation is $\llbracket (A^o\Rightarrow B^o)^o\rrbracket:=
\llbracket A^o\rrbracket \to \llbracket B^o\rrbracket$.
      \item For an expression of the shape $(\forall x^\sigma A^o)^o$
its interpretation is \[\llbracket (\forall x^\sigma
A^o)^o\rrbracket:= \operatorname{inf}\big\{\llbracket
A^o\rrbracket\{x^\sigma:=s\}\ \big|\ s\in\llbracket
\sigma\rrbracket\big\}\]
      \end{itemize}
  \end{enumerate} We say that $\mathcal{A}$ satisfies a sequent
$x_1:A_1^o, \dots, x_k:A^k\vdash p:B^o$ if and only if for all
assignment $\mathfrak{a}$ and for all~$b_1, \dots, b_k\in A$, if $b_1
\leq \llbracket A_1^o\rrbracket, \dots, b_k\leq \llbracket
A^o_k\rrbracket$ then $p\{x_1:=b_1, \dots, x_k:=b_k\}\leq \llbracket
B^o\rrbracket$. In this case we write that:~$\mathcal{A}\models
x_1{:}A_1^o, \dots, x_k{:}A^k\vdash p{:}B^o$.

  A rule: \begin{prooftree} \AxiomC{$S_1$} \AxiomC{$\dots$}
\AxiomC{$S_h$} \RightLabel{\scriptsize(Rule)}
\TrinaryInfC{$S_{h{+}1}$}
\end{prooftree} is said to be \emph{adequate} if and only if for every
$\mathcal{A}\in\mathcal{^KOCA}$, if $\mathcal{A}\models S_1, \dots,
S_h$ then $\mathcal{A}\models S_{h{+}1}$.
\end{definition}
\begin{theorem}\label{theo:LomegaAdeq} The rules of the typing system
appearing in Definition \ref{defi:LomegaTyp}, are adequate.
\end{theorem}
\begin{proof} For {\scriptsize{(ax)}} is evident.

  For the implication rules:
  \begin{itemize}
    \item[{$(\to)_i$}] Assume $\mathcal{A}\models \Gamma, x:A^o\vdash
p:B^o$ where $\Gamma=x_1:A_1^o, \dots, x_k:A_k^o$. Consider an
assignment $\mathfrak{a}$ and~$b_1, \dots, b_k\in A$ such that
$b_i\leq \llbracket A^o_i\rrbracket$. We get:
      \begin{center}
      \begin{tabular}{rcl} $(\lambda^*x p)\{x_1:=b_1, \dots,
x_k:=b_k\}\llbracket A^o\rrbracket$&$=$&$(\lambda^*x p\{x_1:=b_1,
\dots, x_k:=b_k\})\llbracket A^o\rrbracket\leq$\\ &&$p\{x_1:=b_1,
\dots, x_k:=b_k, x:=\llbracket A^o\rrbracket\}\leq$\\ &&$\llbracket
B^o\rrbracket$
      \end{tabular}
      \end{center} the last inequality by the assumption
$\mathcal{A}\models \Gamma, x:A^o\vdash p:B^o$.

      Applying the adjunction property we deduce that~$\ce(\lambda^*x
p)\{x_1:=b_1, \dots, x_k:=b_k\}\leq \llbracket (A^o\Rightarrow
B^o)^o\rrbracket$. Since the above is valid for all the assignments,
we conclude $\mathcal{A}\models \Gamma\vdash \ce(\lambda^*x\
p):(A^o\Rightarrow B^o)^o$.

      \item[{$(\to)_e$}] Assume $\mathcal{A}\models \Gamma\vdash
p:(A^o\Rightarrow B^o)^o$ and $\mathcal{A}\models \Gamma\vdash q: A^o$
where $\Gamma=x_1:A_1^o, \dots, x_k:A_k^o$. Consider an assignment
$\mathfrak{a}$ and~$b_1, \dots, b_k\in A$ such that $b_i\leq
\llbracket A^o_i\rrbracket$. By hypothesis we get:
      \begin{center}
      \begin{tabular}{rcl} $p\{x_1:=b_1, \dots,
x_k:=b_k\}$&$\leq$&$\llbracket A^o\rrbracket \to \llbracket
B^o\rrbracket$\\ &and&\\ $q\{x_1:=b_1, \dots,
x_k:=b_k\}$&$\leq$&$\llbracket A^o\rrbracket$\\
      \end{tabular}
      \end{center} and by monotonicity of the application in
$\mathcal{A}$ we obtain:
        \[pq\{x_1:=b_1, \dots, x_k:=b_k\}\quad \leq\quad (\llbracket
A^o\rrbracket \to \llbracket B^o\rrbracket)\quad \llbracket
A^o\rrbracket\quad \leq\quad\llbracket B^o\rrbracket\] Since the above
is valid for all the assignments, we conclude that $\mathcal{A}\models
\Gamma\vdash pq:\llbracket B^o\rrbracket$.
  \end{itemize} For the quantifiers:
  \begin{itemize}
    \item[{$(\forall)_i$}] Assume $\mathcal{A}\models \Gamma\vdash
p:A^o$ and that $x^\sigma$ does not appear free in $\Gamma$, where
$\Gamma=x_1:A_1^o, \dots, x_k:A^o_k$. Consider an
assignment~$\mathfrak{a}$ and~$b_1, \dots, b_k\in A$ such that
$b_i\leq \llbracket A^o_i\rrbracket$.

      Since $A_1^o, \dots, A_k^o$ does not depend upon $x^\sigma$, by
the assumption $\mathcal{A}\models \Gamma\vdash p:A^o$, we get:
      \begin{center}$p\{x_1:=b_1, \dots, x_k:=b_k\}\leq\llbracket
A^o\rrbracket\{x^\sigma:=s\}$ for all $s\in\llbracket
\sigma\rrbracket$
      \end{center} Then $p\{x_1:=b_1, \dots,
x_k:=b_k\}\leq\operatorname{inf} \{\llbracket
A^o\rrbracket\{x^\sigma:=s\}\ |\ s\in\llbracket
\sigma\rrbracket\}=\llbracket (\forall x^\sigma A^o)^o\rrbracket$. We
conclude as before that $\mathcal{A}\models \Gamma\vdash p:(\forall
x^\sigma A^o)^o$.
      \item[$(\forall)_e$] Assume $\mathcal{A}\models \Gamma\vdash
p:(\forall x^\sigma A^o)^o$, where $\Gamma=x_1:A_1^o, \dots,
x_k:A^o_k$. Consider an assignment $\mathfrak{a}$ and~$b_1, \dots,
b_k\in A$ such that $b_i\leq \llbracket A^o_i\rrbracket$. By the
assumption $\mathcal{A}\models \Gamma\vdash p:(\forall x^\sigma
A^o)^o$ we get:
      \[p\{x_1:=b_1, \dots, x_k:=b_k\}\leq \llbracket
A^o\rrbracket\{x^\sigma:=s\} \mbox{ for all }s\in\llbracket
\sigma\rrbracket\] Since $\llbracket M^\sigma\rrbracket\in\llbracket
\sigma\rrbracket$ we obtain:
      \[p\{x_1:=b_1, \dots, x_k:=b_k\}\leq \llbracket A^o\rrbracket
\{x^\sigma:=\llbracket M^\sigma\rrbracket \}= \llbracket
A^o\{x^\sigma:=M^\sigma\}\rrbracket\] We conclude as before that
$\mathcal{A}\models \Gamma\vdash p:A^o\{x^\sigma:=M^\sigma\}$.
  \end{itemize}
\end{proof} The language of higher-order Peano arithmetic
--$(\operatorname{PA})^\omega$--is an instance of $\mathcal{L}^\omega$
where we distinguish a constant of kind $I$ and two constants of
expression $0^I$ and $\operatorname{succ}^{I\to I}$.
\begin{definition} For each kind $\sigma$ we define the Leibniz equality
$=_\sigma$ as follows:
\[ x_1^\sigma =_\sigma x_2^\sigma :\equiv \forall y^{\sigma\to o} \Big
((y^{\sigma \to o} x_1^\sigma)^o \Rightarrow (y^{\sigma \to
o}x_2^\sigma)^o\Big)^o
\]
\end{definition} The axioms of Peano arithmetic are equalities over the
kind $I$, except for $\forall x^I ((\operatorname{succ}^{I\to I}x^I
=_I 0^I)\Rightarrow \bot)^o$ --which we abbreviate $\forall x^I
(\operatorname{succ}^{I\to I}x^I\neq 0^I)^o$-- and for the induction
principle.

\begin{definition}
Fixed $\mathcal{A} \in ^\mathcal{K}\mathcal{OCA}$ , we say that $a\in
A$ realizes a formula $F^o$ if $a\leq \llbracket 
F^o\rrbracket$. We write $a\Vdash_{\mathcal{A}} F^o$ for ``$a$
realizes $F^o$'', or simply as $a\Vdash F^o$, whenever it does not
cause confusion. 

The theory of
$\mathcal{A}$ is the set of closed formul\ae~ $F^o$ 
such that there is an $a\in \Phi$ which realizes $F^o$. The theory of
$\mathcal{A}$ is denoted by $\operatorname{th}(\mathcal{A})$.  
\end{definition}
In this presentation of Krivine's realizability, the orthogonality is
implicit in the implication $\to$ that is part of the structure of the
$\koca$.
\begin{lemma}\label{lm:equialitiesAreRealized}
Let us consider an equality $M^\sigma=_\sigma N^\sigma$ such that $\llbracket
M^\sigma \rrbracket = \llbracket N^\sigma \rrbracket$. Then the
equality $M^\sigma=_\sigma N^\sigma$ is realized by $\come(\lambda^*x.x)$.
\end{lemma}
\begin{proof}
Consider an $f\in\llbracket \sigma\to o\rrbracket = A^{\llbracket \sigma
\rrbracket}$, since $\llbracket M^\sigma \rrbracket = \llbracket
N^\sigma\rrbracket$ we have $f(\llbracket M^\sigma
\rrbracket)=f(\llbracket N^\sigma \rrbracket)$. We conclude that
$(\lambda^* x.x)\llbracket y^{\sigma\to o}M^\sigma \rrbracket \leq
\llbracket y^{\sigma\to o}M^\sigma \rrbracket = \llbracket
y^{\sigma\to o}N^\sigma \rrbracket$ and $\come (\lambda ^* x.x)\leq \llbracket
y^{\sigma \to o}M^\sigma \Rightarrow y^{\sigma \to
  o}N^\sigma\rrbracket$ for every assignment of
$y^{\sigma\to o}$. Hence $\come (\lambda^*
x.x)\Vdash M^\sigma =_\sigma N^\sigma$. 
\end{proof}
\begin{proposition} In every $\koca$ $\mathcal{A}$ all axioms of Peano
  arithmetic 
  but the induction principle are in $\operatorname{th}(\mathcal{A})$. 
\end{proposition}
\begin{proof} By \ref{lm:equialitiesAreRealized} all the axioms which
  are equalities are realized by $\come (\lambda^*x.x)$. Moreover, the
axiom which say that $0$ is not a successor is also realized:
It is easy verify that $\llbracket \forall
x^I[\operatorname{succ}^{I\to I}x^I=_I 0^I\Rightarrow
\bot]\rrbracket=\llbracket \top \Rightarrow \bot\rrbracket \to
\llbracket \bot\rrbracket$. By monotonicity $\llbracket \top\Rightarrow
\bot\rrbracket\coms\leq \llbracket \top\Rightarrow \bot\rrbracket
\llbracket \top\rrbracket \leq \llbracket \bot\rrbracket$. Thus
$\llbracket \top\Rightarrow \bot\rrbracket \coms\leq \llbracket
\bot\rrbracket$ and hence $\come (\lambda^*x.x\coms)\Vdash \llbracket
\forall x^I[\operatorname{succ}^{I\to I}x^I=_I 0^I\Rightarrow 
\bot]\rrbracket$  
\end{proof}
\begin{definition} The formula $\mathbb{N}(z^I)$ is defined as:
\[ \forall x^{I \to o}(\forall y^I ((x^{I\to o}y^I)^o\Rightarrow
(x^{I\to o}(\operatorname{succ}^{I\to I}y^I))^o\Rightarrow ((x^{I\to
o}0^I)^o \Rightarrow (x^{I\to o}z^I)^o)^o
\]
\end{definition} 
\begin{remark}
Since the equational axioms of Peano arithmetic and the axiom
$\forall x^I[\operatorname{succ}^{I\to I}x^I=_I 0^I\Rightarrow  
\bot]$ are universal formul\ae\ and, therefore, imply their 
relativization to $\mathbb{N}$. The relativization of the induction
principle to $\mathbb{N}$ is $\forall x^I (\mathbb{N}(x^I)\Rightarrow
\mathbb{N}(x^I))$, which is realized by means of $\come
(\lambda^*x.x)$. Thus, relativizing to $\mathbb{N}$ all proofs of higher-order
arithmetic, we find realizers in $\Phi$ for their theorems by means of
adequacy \ref{theo:LomegaAdeq}. In other words,
$\operatorname{th}(\mathcal{A})$ contains $\operatorname{th}((PA)^\omega)$. 
\end{remark}  

\bibliographystyle{plain}

\begin{thebibliography}{1}
\bibitem{Friedman73} Friedman, H. {\em Some applications of Kleene's
  methods for intuitionistic system}, Cambridge summer school in
  mathematics, 337, pp 113-170 (1973)

\bibitem{gierz2003continuous} Gierz,G., Hofmann,K. H., Keimel, K.,
Lawson, J.D., Mislove,M., Scott, D. S.  \newblock {\em Continuous
{L}attices and {D}omains}.  \newblock Cambridge University Press,
2003.

\bibitem{Griffin} Griffin, T.G. \newblock {\em A Formul\ae-as-Types
  Notion of Control.} In Conference Record of the Seventeenth Annual
  ACM Symposium on Principles of Programming Languages, 1990. 
\bibitem{kn:hofstra2006} Hofstra,P. {\em All realizability is
relative}, Math. Proc. Cambridge Philos. Soc. 141 (2006), no. 2,
pp. 239--264.
\bibitem{kn:hofstra2008}Hofstra, P. {\em Iterated realizability as a
comma construction}, Math. Proc. Cambridge Philos. Soc. 144 (2008),
no. 1, pp. 39--51.
\bibitem{kn:hyland} Hyland, J. M. E.{\em The effective topos},
Proc. of The L.E.J. Brouwer Centenary Symposium (Noordwijkerhout 1981)
pp. 165-216, North Holland 1982

\bibitem{kn:tripos} Hyland,J.M.E., Johnstone, P.T., Pitts, A.M. {\em
    Tripos theory}, Math. Proc. Cambridge Phil. Soc. 88 (1980), pp 205--232.
\bibitem{elephant1} Johnstone, P.T.  \newblock {\em Sketches of an
elephant: a topos theory compendium. {V}ol. 1}, volume~43 of {\em
Oxford Logic Guides}.  \newblock The Clarendon Press Oxford University
Press, New York, 2002.

\bibitem{kn:partordcombalg} Hofstra, P., van Oosten, J. {\em Ordered
partial combinatory algebras}, Math. Proc. Cambridge Philos. Soc. 134
(2004), no. 3, pp. 445--463.
\bibitem{Kleene} Kleene, S.C. \newblock{\em On the interpretation of
  intuitionistic number theory.}, Journal of Symbolic Logic,
  10:109-124, 1945.  
\bibitem{KrivineStorageOp} Krivine, J.L., {\em A general storage
  theorem for integers in call-by-name lambda-calculus.},
  Th. Comp. Sc., 129, p. 79-94 (1994).  
\bibitem{kn:krreal} Krivine, J.-L.,{\em Realizability algebras: a
    program to well order $\mathbb R$}, Logical methods in computer
  science, vol 7, (3: 02), 2001, pp 1--47.
\bibitem{kn:kr2001} Krivine, J.-L. {\em Types lambda-calculus in
classical Zermelo-Fraenkel set theory}, Arch. Math. Log. 40 (2001),
no. 3, pp. 189--205.
\bibitem{kn:kr2003} Krivine, J.-L.,{\em Dependent choice, ‘quote’ and
the clock}, Th. Comp. Sc. 308 (2003), pp. 259--276.
\bibitem{kn:kr2008} Krivine, J.-L. {\em Structures de
r\'ealisabilit\'e, RAM et ultrafiltre sur $\mathbb N$},\,(2008). http
: //www.pps.jussieu.fr/~krivine/Ultrafiltre.pdf.
\bibitem{kn:kr2009} Krivine, J.-L. {\em Realizability in classical
logic in Interactive models of computation and program behaviour},
Panoramas et synth\`eses 27 (2009), SMF.

\bibitem{lack20102} Lack, S.  \newblock A 2-categories companion.
\newblock In {\em Towards higher categories}, pages
105--191. Springer, 2010.

\bibitem{lawvere1969adjointness}  Lawvere, F.W.  \newblock
Adjointness in foundations.  \newblock {\em Dialectica},
23(3-4):281--296, 1969.

\bibitem{lawvere1970equality} Lawvere, F. W.  \newblock Equality in
hyperdoctrines and the comprehension schema as an adjoint functor.
\newblock {\em Applications of Categorical Algebra}, 17:1--14, 1970.
\bibitem{McCarty84} McCarty, D. {\em Realizability and recursive
  mathematics}, Technical Report CMU-CS-84-131. Departament of
  Computer Science, Carnegie-Mellon University, 1984. Report version
  of the author's PhD thesis, Oxford University 1983. 
\bibitem{Myhill73} Myhill, J. {\em Some properties of intuitionistic
  Zermelo-Fraenkel set theory}, Lecture notes in mathematics 337,
  pp 206-231 (1973).  
\bibitem{kn:sorensen2006lectures} S{\o}rensen,M.H., Urzyczyn,P. {\em
    Lectures in the Curry--Howard isomorphism} Elsvier, Studies in
  Logic and the foundations of mathematics, vol 149, 2006. 
\bibitem{kn:streicher} Streicher, T. {\em Krivine'\/s Classical
Realizability from a Categorical Perspective}, Math. Struct. in
Comp. Science
\bibitem{kn:vanOosbook} van Oosten, J. {\em Realizability, an
Introduction to its Categorical Side}, (2008), Elsevier.
\end{thebibliography}

\end{document}